\documentclass[11pt]{article}
\usepackage[margin=4.46cm]{geometry}
\usepackage{graphicx}
\usepackage{epsfig}
\usepackage{amsmath}
\usepackage{amssymb}
\usepackage{caption}
\usepackage{amsthm}
\usepackage{enumitem}
\usepackage[most]{tcolorbox}

\usepackage{tikz}
\numberwithin{equation}{section}
\newcommand{\commentout}[1]{}

\def \Rset {{\mathbb R}}

\def \Zset {{\mathbb Z}}

\def \Nset {{\mathbb N}}

\newcommand{\nit}{\noindent}

\newcommand{\be}{\begin{equation}}
\newcommand{\ee}{\end{equation}}
\newcommand{\ba}{\begin{eqnarray}}
\newcommand{\ea}{\end{eqnarray}}
\newcommand{\bi}{\begin{itemize}}
\newcommand{\ei}{\end{itemize}}
\newcommand{\br}{\begin{eqnarray}}
\newcommand{\er}{\end{eqnarray}}

\newtheorem{example}{Example}[section]
\newtheorem{theo}{Theorem}[section]
\newtheorem{defin}{Definition}[section]

\newtheorem{lem}{Lemma}[section]
\newtheorem{cor}{Corollary}[section]
\newtheorem{rmk}{Remark}[section]

\begin{document}
\title{{\Large\bf{Geometry of the Minimal Level Set of the Effective Hamiltonian in Two Dimensions}}}
\author{Yifeng Yu\thanks{Department of Mathematics, University of California, Irvine. Email: \texttt{yifengy@uci.edu}. This work was partly supported by NSF grant DMS-2000191.}}
\date{}
\maketitle

\begin{abstract}
In this paper, we characterize the geometric structure of the boundary of the minimal level set $F_0$ of the effective Hamiltonian $\overline{H}$ associated with the mechanical Hamiltonian
\[
H(p,x)=\frac12|p|^2+V(x)
\]
in dimension $n=2$, where $V$ on $\mathbb{T}^2=\mathbb{R}^2/\mathbb{Z}^2$ has a unique maximum and Hessian at this maximizer has two distinct negative eigenvalues.

For $n=2$, the geometry of the level sets of $\overline{H}$ strictly above the minimum has been largely understood since the 1990s, mainly through the equivalent formulation in terms of stable norms; we fill the remaining gap at the minimal level by providing an explicit, verifiable characterization of $\partial F_0$.

In particular, we show that $p \in \partial F_0$ does not lie on any flat edge if and only if $\partial F_0$ is differentiable at $p$ and its outer normal direction is irrational, except possibly at one exceptional pair of points $\pm p_0$. Consequently,  flat edges are dense along $\partial F_0$. We also construct an example demonstrating that this exceptional pair can occur, showing the result is sharp. 

\end{abstract}


\section{Introduction}

Assume that $H=H(p,x)\in C(\Rset^n\times \Rset^n)$ is $\Zset^n$-periodic in $x$ and uniformly coercive in $p$, i.e
$$
\lim_{|p|\to +\infty}\min_{x\in \Rset^n}H(p,x)=+\infty.
$$
For each $\epsilon>0$, let $u^{\epsilon}\in C(\mathbb{R}^n\times [0,\infty))$ be the viscosity solution to the following Hamilton-Jacobi equation
\be\label{HJ-ep}
\begin{cases}
u_{t}^{\epsilon}+ H\left(Du^{\epsilon}, {x\over \epsilon}\right)=0 \quad &\text{in $\mathbb{R}^n\times (0,\infty)$},\\
u^{\epsilon}(x,0)=g(x) \quad &\text{on $\mathbb{R}^n$}.
\end{cases}
\ee
It was proved by Lions, Papanicolaou and Varadhan \cite{LPV} that $u^{\epsilon}$, as $\epsilon\to 0$, converges locally uniformly to $u$, the solution of the effective equation,
\be\label{HJ-lim}
\begin{cases}
u_t+{\overline H}(Du)=0 \quad &\text{in $\mathbb{R}^n\times (0,\infty)$},\\
u(x,0)=g(x) \quad &\text{on $\mathbb{R}^n$}. 
\end{cases}
\ee
Here $\overline H:\Rset^n\to \Rset$ is called ``effective Hamiltonian" in \cite{LPV}  that is determined by the following cell problem: for any $p\in \Rset^n$, there exists a unique number $\overline H$ such that the equation
\be\label{cell}
H(p+Dv,x)=\overline H(p)
\ee
has a $\Zset^n$-periodic viscosity solution. If \(H\) is convex in the \(p\)-variable, then \(\overline H\) admits the inf--max representation
\[
\overline H(p)
= \inf_{\phi \in C^{1}(\mathbb{T}^{n})}
  \max_{x \in \mathbb{R}^{n}}
  H\bigl(p + D\phi(x), x\bigr).
\]
See, for example, \cite{E2008} for a proof.
 
A natural yet largely open problem is to understand the properties of \( \overline{H} \), which was first explicitly proposed in \cite{E2008}. Similar questions have been studied for decades in the context of stable norms in periodic settings 
(see, for instance, \cite{AB, B2,  Burago1994, BIK}), as well as in the study of limit shapes in first-passage percolation 
in stochastic settings (see the survey \cite{ADH} and the references therein). We would like to mention that the effective Hamiltonian is called the ``$\alpha$ function" in dynamical system literature and the action variable $p$ is often denoted as $c$ (cohomology class): $\alpha(c)$.

For simplicity, throughout this paper we consider the mechanical Hamiltonian
\[
H(p, x) = \frac{1}{2}|p|^2 + V(x),
\]
which serves as one of the most important examples. The potential function \( V \) is assumed to be smooth and \( \mathbb{Z}^n \)-periodic. The following basic global properties were established in \cite{LPV}:
\[
\text{\( \overline{H} \) is convex}, \quad \overline{H}(p) = \overline{H}(-p), \quad \min_{\mathbb{R}^n} \overline{H} = \max_{\mathbb{R}^n} V,
\]
and
\[
\frac{1}{2}|p|^2 + \min_{\mathbb{R}^n} V \;\leq\; \overline{H}(p) \;\leq\; \frac{1}{2}|p|^2 + \max_{\mathbb{R}^n} V.
\]

\medskip

Nevertheless, the macroscopic properties described above are too coarse to capture the delicate local structure of the effective Hamiltonian, which remains far from obvious.

Hereafter, $\Bbb T^n=\mathbb{R}^\backslash\mathbb{Z}^n$ represents the $n$-dimensional flat torus.  Without loss of generality, we assume throughout this paper that
\[
\max_{\mathbb{R}^2} V = 0.
\]
Then $\min_{p\in \Rset^n}\overline H(p)=0$.

\medskip

\noindent\textbf{$\bullet$ Known local finer properties of \( \overline{H} \).}  
Below we summarize the known results concerning the local structure of \( \overline{H} \).  These results are often derived using Aubry–Mather theory. See \cite{TY} for connections between stable norms and effective Hamiltonian. For $c\geq 0$, define
\[
F_c = \left\{ p \in \mathbb{R}^n \;\middle|\; \overline{H}(p) = c \right\}
\]
and
$$
\Rset\Zset^n=\{\lambda v|\ \lambda\in \Rset, \ v\in \Zset^n\}. 
$$
Assume $c>0.$
\medskip

{\bf $\star$  \underline{Two dimension ($n=2$) without genericity assumptions}.} 

\begin{itemize}
    \item The level curve \( F_c \) is of class \( C^1 \) (see \cite{C}) although the function $\overline H$ itself might not be $C^1$;
    \item For any \( p \in F_c \), if the associated normal vector \( q_p \in \mathbb{R}^2 \setminus \mathbb{R} \mathbb{Z}^2 \), then \( p \) does not lie on any line segment of \( F_c \) (see \cite{B2}). Note that the results in \cite{B2} were stated in an equivalent but more general formulation using stable norms associated with the Riemannian metric $g(x)=2(c-V(x))((dx_1)^2+(dx_2)^2)$ on $\Rset^2$. See Section \ref{section:sketch-of-proof} for more details;
    \item  The level curve \( F_c \) contains at least two line segments (symmetric about the origin) unless \( V \) is constant (\cite{B2}).  This result is optimal. See Example \ref{example:separable}. However,  no information is currently available regarding the possible locations of these line segments for a given $V$. 

\end{itemize}

{\bf $\star$  \underline{Two dimension ($n=2$) with genericity assumptions}.} 

\begin{itemize}

\item For generic $V$,  line segments will appear at every point on $\partial F_c$ whose the normal vector belongs to $\Rset\Zset^2$.

\item It was proved in \cite{Yu} that, when $n=2$, for generic $V$, $\overline H$ is highly degenerate, i.e.,  piecewise one dimensional on a dense open set on $\Rset^2$. Hence the local behavior of $\overline H$ could be changed dramatically by the homogenization process.

\end{itemize}

Generic assumptions are often employed in dynamical systems (and more broadly in analysis, PDE, etc) to rule out pathological cases via small perturbations. However, it is in general impossible to tell whether a given $V$ is ``generic". 
\medskip

{\bf $\star$ \underline{Higher dimensions $n\geq 3$}}

\begin{itemize}

\item The \textbf{Aubry--Mather theory} provides a powerful framework for describing action-minimizing curves. 
In the two-dimensional case ($n=2$), the presence of a topological obstruction enhances the control that one has over such curves, 
and this plays a crucial role in establishing the earlier result. 
In contrast, for $n \geq 3$, our understanding is far more limited: the absence of comparably effective tools makes the analysis 
of the effective Hamiltonian $\overline{H}$ much more difficult, and properties that hold in two dimensions---such as 
differentiability and strict convexity---may no longer be valid (see \cite{BIK, JTY}). 

\item It was proved in \cite{EG} that $\overline{H}$ is strictly convex along any direction not tangent to the level set $F_c$ . 
As a consequence, line segments can only occur within level sets, regardless of the dimension.  

\end{itemize}

\medskip

 To complete the picture in two dimension, we now turn to the minimal level set---an interesting feature of $\overline{H}$ not covered in the study of stable norms since the associated metric $g(x)=2(\max_{\Rset^n}V-V(x))((dx_1)^2+(dx_2)^2)$ on $\Rset^2$ is degenerate at the maximizer(s) of $V$.

 The minimal level $c=0$ is qualitatively different from the supercritical regime $c>0$. When $c>0$, all action-minimizing curves (equivalently, minimizing geodesics) are unbounded in both time directions and cannot intersect twice. When $c=0$, however, the presence of stagnation points (the maximizer(s) of $V$) changes the picture: action-minimizing curves may be finite (e.g., homoclinic), or they may be bounded in one time direction while tending to a stagnation point. For our purposes, we need to glue together such pieces, which leads to the failure of the ``no-two-intersections'' property and other intersection restrictions available in the $c>0$ case. See Section~\ref{section:sketch-of-proof} for more details.

  Our goal of this paper is to characterize the geometric properties of  the minimum level
$$
F_0=\{p\in \Rset^n|\ \overline H(p)=\max_{\Rset^n}V\}.
$$
\begin{itemize}
  \item \textbf{Existence of a flat part and explicit examples.}
  It was first shown in \cite{LPV}, via a very simple argument, that $F_0$ is an $n$-dimensional convex set (a flat part) for quite general $V$.
  Subsequently, \cite{Con} gives a detailed analysis of the assumptions on $V$ that ensure the existence or non-existence of a flat part, and it examines several explicit and illuminating examples of $\overline{H}$ in the case $V=\infty$ on a convex subset (the “obstacle”) of various shapes, such as a disk or a square. In addition, \cite{ZZ2012} provides an example showing that $\overline{H}$ is not differentiable at some points of $\partial F_0$. 

  \item \textbf{Dynamical properties along $\partial F_0$.}
 In \cite{Cheng}, a detailed description of Aubry--Mather sets on $\partial F_0$ is given in generic situations; these sets play a crucial role in the construction of Arnold diffusion. Our work is strongly inspired by \cite{Cheng}, and we will explain the precise connections in more detail later in the paper.
\end{itemize}

\begin{tcolorbox}[
  enhanced,
  breakable,
  colback=white,
  colframe=black,
  boxrule=0.6pt,
  arc=2mm,
  left=1.6mm,right=1.6mm,top=1.2mm,bottom=1.2mm
]
\textbf{Assumption (M).}
We assume that the potential $V$ is smooth and $\Zset^2$-periodic, and that it attains its maximum uniquely on $\mathbb{T}^2$. More precisely,
\[
\{x\in\Rset^2:\ V(x)=0\}=\Zset^2.
\]
Moreover, the Hessian of $V$ at the origin $O=(0,0)$ has two distinct eigenvalues $-a^2$ and $-b^2$ with $0<a<b$, i.e.,
\[
\text{$D^2V(O)$ has eigenvalues $-a^2$ and $-b^2$, with $0<a<b$.}
\]
\end{tcolorbox}

Let \( \{v_a, -v_a\} \) and \( \{v_b, -v_b\} \) denote the pairs of unit eigenvectors associated with the eigenvalues \( -a^2 \) and \( -b^2 \), respectively.

Our results can be extended appropriately to the case where \( V \) has multiple non-degenerate maximum points or when $|p|^2$ is replaced by more general quadratic term $p\cdot A(x)\cdot p$.  For clarity of presentation, we adopt the above assumptions throughout this paper.

Let $S^1\subset \Rset^2$ be the set of unit vectors. For $p\in \partial F_0$, write 
$$
n_p=\{q\in S^1|\ \text{$q\cdot (p'-p)\leq 0$ for all $p'\in F_0$}\}.
$$
as the unit outer normal cone at $p$. We say that $\partial F_0$ is differentiable at $p$ if $n_p$ contains only one element $n_p=\{q_p\}$, equivalently, there is a unique supporting line (tangent line) at $p$. The vector $q_p$ is then the outer unit normal vector at $p$.

\begin{defin}
$p\in \partial F_0$ is called a `` linear point" if there exists $p'\in \partial F_0$ such that $p\not=p'$ and the line segment (flat edge)
$$
\{tp+(1-t)p'|\ t\in [0,1]\}\subset \partial F_0.
$$
Otherwise, $p$ is called a  ``nonlinear point".  $p$ is called a rational nonlinear point if $p$ is a nonlinear point and $n_p=\{q_p\}$ for some $q_p\in \Rset\Zset^2$.
\end{defin}
Since  $\overline H$ is even in $p$, a point $p$ is nonlinear if and only if $-p$ is nonlinear.

\medskip

{\bf Main Contribution:} We show that $p \in \partial F_0$ is a nonlinear point if and only if $\partial F_0$ is differentiable at $p$ and its outer normal direction is irrational, except possibly at one exceptional pair of points $\pm p_0$. Example \ref{example:removable} demonstrates that such a pair can indeed occur. Consequently, our results are sharp.

\medskip

The following two theorems give detailed and precise description of the above statement. 
\begin{theo}\label{main} Given $n=2$ and {\bf Assumption (M)},  then one of the following holds:

\medskip

(1) For any $p\in \partial F_0$, $p$ is a nonlinear point {\bf if and only if} $\partial F_0$ is differentiable at $p$ and the associated outer unit normal vector $q_p\in S^1\backslash \Rset\Zset^2$; or

(2) There exists  a rational nonlinear point $p_0\in \partial F_0$ such that for any $p\in \partial F_0\backslash \{p_0,-p_0\}$, $p$ is a nonlinear point {\bf if and only if} $\partial F_0$ is differentiable at $p$ and the associated outer unit normal vector $q_p\in S^1\backslash \Rset\Zset^2$.

\end{theo}

Note that the above result implies that flat edges are dense on $\partial F_0$.

In addition, the following theorem says that any non-differentiable point on $\partial F_0$ must be a vertex of cone lying on $\partial F_0$.  

\begin{theo}\label{main2} Given $n=2$ and {\bf Assumption (M)}, if $\partial F_0$ is not differentiable at $p\in \partial F_0$, then there exist two non-parallel integer vectors ${v_0},{v_1}\in\Zset^2$ and $\delta>0$ such that the following two hold:
\begin{enumerate}
\item The cone
\be\label{eq:coneequation}
\bigl\{\, p + t\,(-1)^i\,\tfrac{v_i^{\perp}}{\|v_i^{\perp}\|} \;:\; i=0,1,\ t\in[0,\delta] \,\bigr\}
\subset \partial F_0,
\ee
where $v_i^{\perp}$ denotes the $90^\circ$ counterclockwise rotation of $v_i$, and
\item These two integer vectors $v_0$ and $v_1$ form a positively oriented unimodular pair, i.e,
\[
\det A=1,
\]
where $A=(v_1^{\mathsf{T}}\,\, v_0^{\mathsf{T}})$ is the $2\times 2$ matrix whose columns are $v_1^{\mathsf{T}}$ and $v_0^{\mathsf{T}}$.
\end{enumerate}
\end{theo}

\begin{rmk}\label{rmk:example}  In  section 6, we have constructed three representative examples demonstrating the sharpness of our results. 

\medskip

$\bullet$ Example \ref{example:nonremovable}:  $\partial F_0$ is $C^1$ and case (1) in Theorem \ref{main}  occurs. In addition, $\overline H$ is nowhere differentiable on $\partial F_0$, which can be easily extended to all dimensions. The example is stable under small perturbation of $V$. See \cite{ZZ2012} for another stable non-differentiability example.

$\bullet$ Example \ref{example:removable}: Case 2 in Theorem \ref{main} occurs. This is a rare situation. 

$\bullet$ Example \ref{example:separable}:  $\partial F_0$ is a rectangle. In this case, $\partial F_0$ has no nonlinear points. Note that Theorem \ref{main}  implies that $\partial F_0$ has at least four flat edges.  In additin,  a special case also provides an example where $F_c$ contains exactly two flat edges for $c>0$.

\end{rmk}
\medskip

{\bf Outline of the Paper.} In Section~2, we review some basic definitions and properties from Aubry–Mather theory and weak KAM theory. We also outline the overall proof strategy and highlight the main novelty of our approach compared to existing results  about $F_c$ when $c>0$. Section~3 is devoted to the proof of Theorem~\ref{main}, while Section~4 contains the proof of Theorem~\ref{main2}. In Section~5, we establish a local approximation result for minimizing curves near the maximum point $(0,0)$, which plays a central role in the proof of Theorem~\ref{main}. Section~6 presents the three examples mentioned in the preceding Remark \ref{rmk:example}. Finally, in Section~7 (Appendix A), we provide an alternative proof that \( F_0 \) is an \( n \)-dimensional convex set when $V$ has finitely many maximum points by constructing subsolution to (\ref{mech-cell}). In Appendix B and C, we establish several technical results used throughout the paper.

\section{Preliminary} For readers' convenience, we give a brief review about some relevant knowledge of Hamilton-Jacobi equations, Aubry-Mather theory and the weak KAM theory. See \cite{W-E, EG, F, Lions, Tran}  for more details. Our approach will mainly be from the point of view of PDEs.  Let $\Bbb T^n=\Rset^n/\Zset^n$ be the $n$-dimensional flat torus and $H(p,x)\in C^{\infty}(\Rset^n\times \Rset^n)$ be a Hamiltonian satisfying
\begin{itemize}
\item[(H1)] (Periodicity) $x\mapsto H(p,x)$ is $\Zset^n$-periodic;

\item[(H2)] (Uniform convexity) There exists $c_0>0$ such that for all $\eta=(\eta_1,...,\eta_n)\in \Rset ^n$,
and $(p,x)\in \Rset^n \times \Rset^n$,
$$
\sum_{i,j=1}^{n}\eta_{i}{\partial ^2H\over \partial p_i\partial
p_j}\eta_{j}\geq c_0|\eta|^2.
$$
\end{itemize}
Denote by
$$
L(q,x)=\sup_{p\in \Rset^n}\{q\cdot p-H(p,x)\}
$$
the Lagrangian associated with $H$. 

The following is a basic and well known property.  For reader's convenience, we present its proof here. 

\begin{lem}\label{calib} Let $U$ be an open subset of $\Rset^n$. Assume that for some $c\in \Rset$, $w\in W^{1,\infty}(U)$ satisfies that 
$$
H(Dw,x)\leq c \quad \text{for a.e $x\in U$},
$$
i.e., $w$ is a viscosity subsolution to $H(Dw,x)=c$.  Then for any $\eta\in AC([t_1,t_2], U)$,
$$
\int_{t_1}^{t_2} (L(\dot\eta(t),\eta (t))-c)\,dt\geq w(\eta(t_2))-w(\eta(t_1)).
$$
Here $ AC([a,b], S)$ stands for the set of absolutely continuous curves $[a,b]\to S$. 

\end{lem}

Proof:  By mollification and approximation,  we may assume that $w$ is $C^1$.  Then
$$
L(\dot \xi(t), \xi(t))+H(Dw(\xi(t),\xi(t))\geq  Dw(\xi(t))\cdot \dot \xi(t).
$$
Taking integration on both sides leads to the conclusion.\qed

\subsection{Global characteristics}

Given $p\in \Rset^n$, let $v$ be a $\Zset^n$-periodic viscosity subsolution  to (\ref{cell}), i.e., $H(p+Dv, x)\leq \overline H(p)$ for a.e $x\in \Rset^n$.  We say that a Lipschitz continuous curve $\gamma:\Rset\to \Rset ^n$ is a {\it global characteristics} associated with $v$  if for all $t_1<t_2$ and $u=p\cdot x+v$, 
$$
\int_{t_1}^{t_2}L(\dot \gamma,\gamma)+\overline H(p)\,ds=u(\gamma(t_2))-u(\gamma(t_1)).
$$
\noindent
Note that this is equivalent to stating that $\gamma$ is {\it $(u, L, \overline{H}(p))$-calibrated} in the sense of~\cite{F}. Here, however, we choose to  use standard PDE terminology. Also, if we restrict $t_1,t_2$ to an given interval, then $\gamma$ is a local characteristics. 

Let
$$
J_v=\{\text{global characteristics associated with $v$}\}.
$$
According to the standard theory in Hamilton-Jacobi equation (\cite{Lions}),  $v$ is $C^1$ along a characteristics and 
$$
\dot \gamma(t)=D_pH(p+Dv(\gamma(t)),\gamma(t)) \quad \text{for all $t\in \Rset$.}
$$
See also \cite{F}.

Hence,  it is natural to understand the structure of global characteristics.  A curve $\xi: \Rset\to \Bbb T^n$ is called an {\it absolutely minimizing curve with respect to
$L(q,x)+\overline H(p)$} in $\Rset ^n$, i.e., for any
$-\infty<s_2<s_1<\infty$, $-\infty<t_2<t_1<\infty$ and $\eta\in AC([s_1,s_2], \Rset^n)$ subject to $\eta (s_2)=\xi (t_2)$ and $\eta (s_1)=\xi (t_1)$
the following inequality holds, 
\be\label{abs}
\int_{s_1}^{s_2} \left( L(\dot \eta(s),\eta(s))+\overline H(p)\right)\,ds\geq 
\int_{t_1}^{t_2} \left( L(\dot\xi(t),\xi (t))+\overline H(p) \right)\,ds.
\ee

Here are facts which are well known to experts.

\medskip

\textbf{(P1)} By Lemma~\ref{calib}, every global characteristic $\gamma$ is an absolutely minimizing curve and therefore satisfies the Euler–Lagrange equation:
\[
\frac{d}{dt} \big( D_q L(\dot{\gamma}(t), \gamma(t)) \big) = D_x L(\dot{\gamma}(t), \gamma(t)) \quad \text{for all } t \in \mathbb{R}.
\]
This relationship is one of the key reasons behind the connection between the cell problem~\eqref{cell} and Aubry–Mather theory, which studies the structure of absolutely minimizing curves (or minimizing geodesics).

\medskip

{\bf (P2)} Two different absolutely minimizing curves can not intersect twice. 

\subsubsection{More Properties for $H(p,x)={1\over 2}|p|^2+V(x)$}

In this section, we define several key quantities and establish two properties of the characteristics that will be used in the subsequent analysis. The corresponding cell problem becomes 
\be\label{mech-cell}
{1\over 2}|p+Dv|^2+V(x)=\overline H(p).
\ee

For $x, y\in \Rset^2$ and $p\in \Rset^2$, we define
$$
h_p(x,y)=\inf_{ \substack{\eta\in AC([0,t], \Rset^2),\ t\geq 0 \\ \xi(0)=x,\ \xi(t)=y}}\int_{0}^{t}\left({1\over 2}|\dot \eta|^2-V(\eta)+\overline H(p)\right)\,ds
$$
and
$$
h(x,y)=\inf_{ \substack{\eta\in AC([0,t], \Rset^2),\ t\geq 0 \\ \xi(0)=x,\ \xi(t)=y}}\int_{0}^{t}\left({1\over 2}|\dot \eta|^2-V(\eta)\right)\,ds.
$$

\medskip

The following properties of $h_p(x,y)$ are obvious. We leave the proof to readers. 
\begin{lem}\label{h-bound}

(1)For $x, y\in \Rset^2$ and $p\in \Rset^2$,
$$
h_p(x,y)\geq p\cdot (y-x)+v(y)-v(x),
$$
for any viscosity subsolution $v$ to the cell problem (\ref{mech-cell}).

(2) For every $\delta>0$,  
$$
\omega({\delta})=\inf_{\{x,y, p\in \Rset^2, \ |x-y|\geq \delta\}}h(x,y)>0.
$$

(3) If \( x \notin \mathbb{Z}^2 \), then there exists \( r=r_x > 0 \) such that for any \( y \in B_r(x) \), there exist \( T > 0 \) and \( \eta \in AC([0, T], \mathbb{R}^2) \) satisfying $\eta(0) = x$, $\quad \eta(T) = y$ and
$$
h_p(x,y)=\int_{0}^{T}\left({1\over 2}|\dot \eta|^2-V(\eta)+\overline H(p)\right)\,ds.
$$
\end{lem}

\begin{lem}\label{lem:lower-bound} Given $p\in \Rset^2$,  let $\xi: \Rset\to \Rset^2$ be a global characteristics associated a viscosity subsolution $v$ to (\ref{cell}). Then for any $\epsilon>0$, there exists a constant $\delta_{\epsilon}$ depending only on $\epsilon$ and $V$ such that for all $t_1<t_2$, 
$$
 \int_{t_1}^{t_2}|\dot \xi(s)|\,ds\geq \epsilon \quad \Rightarrow \quad   u(\xi(t_2))-u(\xi(t_1))\geq \delta_\epsilon.
$$
Here $u(x)=p\cdot x+v(x)$. Note that the left hand side is the length of the curve between $\xi(t_1)$ and $\xi(t_2)$. 
\end{lem}

Proof: Fix $\epsilon > 0$ and argue by contradiction. If the conclusion is false, then by applying suitable translations in time and space, there exists a sequence of vectors $\{p_m\}_{m \geq 1}$, associated viscosity subsolutions $v_m$ to~\eqref{cell}, and corresponding global characteristics $\xi_m$ such that
\[
\xi_m(0) \in \left[-\tfrac{1}{2}, \tfrac{1}{2}\right]^2,
\]
and for some $t_m>0$
$$
 \int_{0}^{t_n}|\dot \xi_n(s)|\,ds\geq \epsilon \quad \mathrm{and} \quad   \int_{0}^{t_m}\left({1\over 2}|\dot \xi_m|^2-V(\xi_m)+\overline H(p_m)\right)\,ds\leq {1\over m} .
$$
By (1) in Lemma \ref{h-bound}, for $0<\theta_1<\theta_2<t_m$,
$$
h_{p_n}(\xi_m(\theta_1), \xi_m(\theta_2))=\int_{\theta_1}^{\theta_2}\left({1\over 2}|\dot \xi_m|^2-V(\xi_m)+\overline H(p_m)\right)\,ds.
$$
Since
$$
{1\over 2}|\dot \xi_m|^2-V(\xi_m)+\overline H(p_m)\geq \sqrt{2}\sqrt{\overline H(p_m)-V(\xi_m)}|\dot\xi_m(t)|,
$$
we have that 
\be\label{zero-limit}
\lim_{m\to \infty}\overline H(p_m)=0  \quad \mathrm{and} \quad \lim_{m\to \infty}\xi_m([0,t_m])=(0,0).
\ee
Let 
$$
h_m(t)=|\dot \xi_m(t)|^2.
$$
Then
$$
\ddot \xi_m=-DV(\xi_m)\ \Rightarrow \ h_m'(t)=2\ddot \xi_m(t)\cdot \dot \xi_m(t)=-2DV(\xi_m(t))\cdot \dot \xi_m(t).
$$
When $m$ is large enough,  (\ref{zero-limit}) implies that
$$
h_m''(t)=-2\dot \xi_m(t)\cdot D^2V(\xi_m(t))\cdot \dot \xi_m(t)+2|DV(\xi_m(t))|^2\geq a^2h_m(t).
$$
Due to Lemma \ref{exp-decay-1}, 
$$
h_m(t)\leq h_m(0){e^{-t\sqrt{a}}}+h_m(t_m)e^{-(t_m-t)\sqrt{a}}.
$$
Hence
$$
|\dot \xi_m(t))|=\sqrt{h_m(t)}\leq  \sqrt{h_m(0)}{e^{-{t\sqrt{a}\over 2}}}+\sqrt{h(t_m)}e^{-{(t_m-t)\sqrt{a}\over 2}}.
$$
Therefore
$$
\int_{0}^{t_m}|\dot \xi_m(s)|\,ds\leq  {2\over \sqrt{a}}\left(\sqrt{h_m(0)}+\sqrt{h_m(t_m)}\right).
$$
Since 
$$
{1\over 2}h_m(t)+V(\xi_m(t))=\overline H(p_m),
$$
(\ref{zero-limit}) leads to 
$$
\lim_{m\to \infty}\int_{0}^{t_m}|\dot \xi_m(t))|\,dt=0,
$$
which is absurd.  \qed

As an immediate corollary,  we can derive that the length of a global characteristics between two points is controlled by their Euclidean distance.
\begin{cor}\label{cor:lower-gap} Given $p\in \Rset^2$.  Let $\xi: \Rset\to \Rset^2$ be a global characteristics associated a viscosity subsolution $v$ to (\ref{mech-cell}). Then  for all $t_1<t_2$, 
$$
M|\xi(t_2)-\xi(t_1)|\geq u(\xi(t_2))-u(\xi(t_1))\geq {\delta_1(l_{t_1,t_2}-1)}.
$$
Here  $\delta_1$ is from Lemma \ref{lem:lower-bound} when $\epsilon=1$, $l_{t_1,t_2}=\int_{t_2}^{t_1}|\dot \xi(s)|\,ds$ and $M=\sqrt{2(\overline H(p)-\min_{\Rset^n}V)}$.
\end{cor}

Proof: The first $\geq$ is obvious. We only need to establish the second $\geq$.  For simplicity of notations, let $t_1=0$ and $t_2=T$.   If $l_{t_1,t_2}\leq 1$, the conclusion is obviously true.  So we assume that $l_{t_1,t_2}>1$.  

Choose $0=s_0<s_1<s_2,...<s_m\leq T$ such that  $\int_{s_i}^{s_{i+1}}|\dot \xi(s)|\,ds=1$ for $i=0,1,2,..m$ and $m=\lfloor l_{t_1,t_2}\rfloor$ (the integer part). 

Then owing to the Lemma \ref{lem:lower-bound}, for $u(x)=p\cdot x+v(x)$, we have that
$$
u(\xi(s_{i+1}))-u(\xi(s_{i}))\geq \delta_1. 
$$
Thus, 
$$
u(\xi(T))-u(\xi(0))\geq m\delta_1.
$$
Hence our conclusion holds. 
\qed

\subsection{Aubry Set}

Some global characteristics play more important roles than others. For $p\in  \Rset^n$,  $\xi:\Rset\to \Rset^n$ is called a {\it universal global characteristic associated with $p$} if it is a global characteristics for every viscosity subsolution to  (\ref{mech-cell}). Let $\mathcal{U}_p$ be the collection of all universal characteristics. Then
\[
\widetilde{\mathcal A}_p
:= \bigcup_{\gamma \in \mathcal{U}_p}\ \bigl\{(\dot\gamma(t),\gamma(t)) : t\in\mathbb{R}\bigr\}\subset \Rset^n\times \Rset^n
\]
is  the Aubry set, which is closed and Euler-Lagrangian flow invariant. We refer the reader  to (\cite{F,FS2004}) for properties of Aubry sets discussed in this section. Throughout this paper, we lift Aubry and Mather sets to $\Rset^n$ for convenience.  Clearly,  the following graph property holds: for every  viscosity subsolution to  (\ref{mech-cell}),
\be{}\label{graphsupport}
\widetilde {\mathcal{A}}_p \subset \{(q,x)\in \Rset ^n\times \Rset^n \,:\,
\text{$Dv(x)$ exists and $p+Dv(x)=D_{q}L(q,x)$}\}.
\ee
Write $\mathcal {A}_p$ as the projection of $\tilde {\mathcal{A}}_p$ on $\Rset^n$.

Below is a well known property of the projected Aubry set.

\medskip

{\bf (P3)}  Given $p\in \Rset^n$, if $v$ is a  viscosity subsolution to (\ref{cell}), then $v$ is differentiable on $\mathcal{A}_p$ and
\be\label{eq:aubrydiff}
H(p+Dv,x)=\overline H(p) \quad \text{for $x\in \mathcal{A}_p$}. 
\ee

\medskip

 The following is the equivalent characterizations of the Aubry set (see \cite{FS2004} for instance).  Let 
$$
G_{t,p}(x,x)=\inf_{\substack{\xi\in AC([0,t],\Rset^n),\\ \xi(0)=x,\  \xi(t)\in x+\Zset^n}}\int_{0}^{t}\left(L(\dot \xi(s),\xi(s))-p\cdot \dot \xi(s)+\overline H(p)\right)\,ds.
$$
Then
$$
\mathcal {A}_p=\{x\in \Rset^n|\  \liminf_{t\to +\infty}G_{t,p}(x,x)=0\},
$$
Thanks to Lemma \ref{calib}, for all $x\in \Rset$ and a viscosity subsolution $v$ to (\ref{cell}), 
\be\label{eq:G-positive}
G_{t,p}(x,x)\geq v(x)-v(x)=0 \quad \text{for all $t\geq 0$}.
\ee

The following result is well known to experts. Since its proof is quite simple,  for reader's convenience, we present it here.

\medskip

\begin{lem}\label{lem:equal-aubry} Given two distinct points $p_0,p_1\in \Rset^n$,   suppose that 
$$
\overline H(p_\lambda)\equiv \overline H(p_0) \quad \text{for all $\lambda\in [0,1]$}
$$
for $p_\lambda=\lambda p_1+(1-\lambda)p_0$. Then for all $\lambda\in (0,1)$,
$$
\widetilde {\mathcal{A}}_{p_\lambda}\subseteq \widetilde {\mathcal{A}}_{p_0}\cap \widetilde {\mathcal{A}}_{p_1}.
$$
Consequently,  for all $\lambda\in (0,1)$
$$
 \widetilde{\mathcal{A}}_{p_\lambda}= \widetilde{\mathcal{A}}_{p_{1\over 2}}
$$
\end{lem}

\medskip

Proof: By the definition of $G_{t,p}(x,x)$, we have that
$$
G_{t,p_{\lambda}}(x,x)\geq \lambda G_{t,p_0}(x,x)+(1-\lambda)G_{t,p_1}(x,x).
$$
Accordingly, by (\ref{eq:G-positive}), 
$$
\liminf_{t\to \infty}G_{t,p_{\lambda}}(x,x)=0 \quad \Rightarrow \quad \liminf_{t\to \infty}G_{t,p_0}(x,x)=\liminf_{t\to \infty}G_{t,p_1}(x,x)=0.
$$
This implies that 
\be\label{eq:projectcontain}
 {\mathcal{A}}_{p_\lambda}\subset {\mathcal{A}}_{p_0}\cap {\mathcal{A}}_{p_1}.
\ee
Now let $v_0$ and $v_1$ and viscosity subsolutions to (\ref{mech-cell}) corresponding to $p_0$ and $p_1$ respectively. Then $v=\lambda v_1+(1-\lambda)v_0$  is a viscosity subsoltuion to (\ref{cell}) corresponding to $p_\lambda$. Thanks to (\ref{eq:aubrydiff}),  $v_0$, $v_1$ and $v_\lambda$ are all differentiable on ${\mathcal{A}}_{p_\lambda}$ and 
$$
H(p_\lambda+Dv_\lambda,x)=H(p_0+Dv_0,x)=H(p_1+Dv_1,x)= \overline H(p_0).
$$
Therefore, strong convexity of $H$ leads to 
$$
p_\lambda+Dv_\lambda(x)=p_0+Dv_0(x)=p_1+Dv_1(x) \quad \text{for $x\in {\mathcal{A}}_{p_\lambda}$}.
$$
Thus, combining (\ref{graphsupport}) and (\ref{eq:projectcontain}), our conclusion $\widetilde {\mathcal{A}}_{p_\lambda}\subseteq \widetilde {\mathcal{A}}_{p_0}\cap \widetilde {\mathcal{A}}_{p_1}$ holds. \qed

Below is a fact that will be used to simplify our proof of Theorem~\ref{main}.
\begin{rmk}\label{rmk:upperaubry}
By the results of \cite{Ber2010} and \cite{FFR2009}, the Aubry set depends upper semicontinuously on the Hamiltonian (equivalently, on the Lagrangian) when $n=2$ and the Hamiltonian is $C^\infty$. 
We note, however, that this property does not hold in general, as shown by a counterexample of Mather in \cite{Mather2004} when the Hamiltonian has lower regularity.
\end{rmk}

\subsection{Mather set}

Denote by $\mathcal{W} $ the set of all Borel probability
measures on $\Rset^n \times \Bbb T^n$ which are Euler-Lagrangian
flow invariant. For fixed $p\in \Rset^n$, $\mu\in \mathcal{W}$ is called a {\it ``Mather measure"} if 
$$
\int_{\Rset^n\times \Bbb T^n} (L(q,x)-p\cdot q)\,d\mu=\min_{\nu\in \mathcal{W}}\int_{\Rset^n\times \Bbb T^n}(L(q,x)-p\cdot q)\,d\nu.
$$
Here $L(q,x)=\sup_{p\in \Rset^n}\{p\cdot q-H(p,x)\}$ is the Lagrangian associated with $H$. Denote by $\mathcal{W}_p$ the set of all such Mather measures. The value of the minimum action on the right hand side turns out to be $-\overline H(p)$, i.e., 
\be\label{negative-hbar}
\min_{\nu\in \mathcal{W}}\int_{\Rset^n\times \Bbb T^n}(L(q,x)-p\cdot q)\,d\nu=-\overline H(p).
\ee

The {\it Mather set}  is defined to be the closure of the union of the support of all Mather measures, i.e., 
$$
\widetilde{\mathcal{M}}_p=\overline {\bigcup_{\mu \in \mathcal{W}_p}\mathrm{supp}(\mu)}.
$$
Hereafter, we lift $\widetilde{\mathcal{M}}_p$ to $\Rset^n$.  It is known that 
$$
\widetilde{\mathcal{M}}_p\subseteq \widetilde{\mathcal{A}}_p,
$$
which implies that the Aubry set is not empty.  The projected Mather set $\mathcal {M}_{p}$ is the projection of $\widetilde{\mathcal {M}}_p$ to $\Rset^n$ (the spacial variable). 

\medskip

 $\mathcal {M}_{p}$ is known to be the set of uniqueness for viscosity solution to the cell problem (\ref{cell}), i.e.,  for two solutions $v_1$ and $v_2$
\be\label{eq:uniquenessset} 
v_1=v_2 \quad \text{on $\mathcal {M}_{p}$}\quad \Rightarrow v_1=v_2 \quad \text{on $\Rset^n$}.
\ee

Hereafter,  a curve $\xi: \Rset\to \Rset^n$ is called an orbit on $\mathcal {M}_p$ (or\  ${\mathcal{A}}_p$) if 
$$
\{(\dot \xi(t),\ \xi(t))\}_{t\in \Rset}\subseteq \widetilde{\mathcal {M}}_p\ (\text{or $\widetilde{\mathcal{A}}_p$}).
$$ 
In particular, due to the graph property (\ref{graphsupport}),  two orbits on $\mathcal{A}_p$ are either disjoint or coincide up to a time translation.

In addition,  a curve $\xi: \Rset\to \Rset^n$ is called {\it  periodic} (modulo $\Zset^n$)  if there exists $v\in \Zset^n$ and $T>0$ (a period of $\xi$) such that $\xi(t+T)-\xi (t)=v$ for all $t\in \Rset$. Note that such a  $\xi$ is periodic when it is projected to the flat torus $\Bbb T^n=\Rset^n\backslash \Zset^n$.

\medskip

\subsection{Sketch of the Proof: Comparison between $F_c$ and $\partial F_0$}\label{section:sketch-of-proof}

Hereafter we focus on $n=2$. Due to the restriction of two dimensional topology on the plane and the crucial property {\bf (P2)},  the classical Aubry-Mather theory has provided  detailed information about the structures of  absolutely minimizing curves. First, we recall a crucial property that is the building block of  the classical Aubry-Mather theory.
\medskip

\noindent\textbf{(P4)} For \( c > 0 \) and \( p \in F_c \), if the unit outward normal vector \( q_p \in \mathbb{R} \mathbb{Z}^2 \), then every orbit in \( \mathcal{M}_p \) is periodic (modulo \( \mathbb{Z}^n \)) and has a first homology class \( (m, n) \) in the same direction as \( q_p \). In addition, $(m,n)$ is a primitive integer vector; that is, either $(m,n)=(\pm1,0)$, $(0,\pm1)$, or $\gcd(|m|,|n|)=1$.  See \cite{B1} for instance.

\medskip

As it was mentioned in the introduction, for $c>0$, $F_c$ must contain line segments unless $V$ is constant \cite{B2}.  Below we provide a sketch of  the proof of this fact, as its reasoning is particularly  relevant.

\medskip

{\bf Step I:}  For $p\in F_c$, if the unit outward normal vector $q_p\in \Rset \Zset^2$ and $p$ does not lie on a line segment on $F_c$, then  the entire plane is filled by periodic orbits on $\mathcal{M}_p$:
\be\label{fill}
\mathcal {M}_p=\mathbb {R} ^2.
\ee
Equivalently, there is no gap between periodic orbits.

Here, we present a simplified proof based on the upper semicontinuity of the Aubry set in two dimensions (see Remark~\ref{rmk:upperaubry}).

 \begin{proof}
Suppose there exists a gap in \( \mathcal{M}_p \) bounded by two periodic (modulo \( \mathbb{Z}^2 \)) orbits $\xi_1$ and $\xi_2$. Assume that  \( p \) is not a linear point.

\begin{itemize}
    \item \textbf{(Approximation)} Choose two sequences of points $\{p_{m,1}\}_{m\geq 1}$ and $\{p_{m,2}\}_{m\geq 1}$ approaching $p$ from two directions (clockwise and counterclockwise along $F_c$). For $i=1,2$, let $\gamma_{m,i}$ be an orbit on $\mathcal{A}_{p_{m,i}}$. Owing to \cite{B1,B2}, the orbits $\gamma_{m,i}$ converge to limiting orbits $\gamma_1$ and $\gamma_2$, which approach two periodic orbits in distinct patterns.  See Figure \ref{f1} below. 

    \item \textbf{(Upper semicontinuity)} By the upper semicontinuity of the Aubry set, both $\gamma_1$ and $\gamma_2$ lie in the Aubry set $\mathcal{A}_p$. Hence, they must be disjoint, which contradicts the topology of two dimensions. The original proof of this step in \cite{B2} relies on constructing viscosity solutions to the cell problem~\eqref{mech-cell}, leading to a contradiction showing that $p$ is a nonlinear point. The existence of gaps essentially ensures the admissibility condition introduced in \cite{Lions}. For a more detailed explanation from the PDE perspective, see \cite{Yu}.
\end{itemize}

\begin{center}
\includegraphics[scale=0.5]{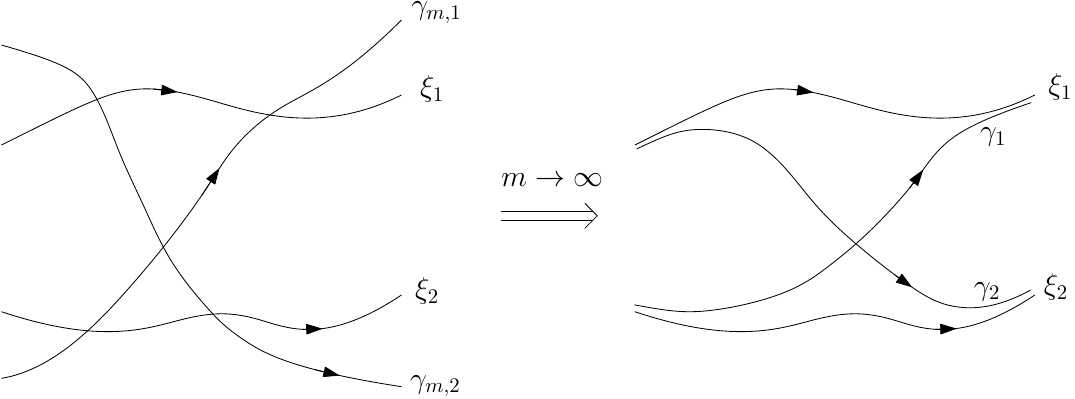}
\captionof{figure}{}
 \label{f1}
\end{center}

\medskip

{\bf Step II:} Accordingly, to establish the existence of line segment at $p\in F_c$ with $n_p\in \Rset\Zset^2$, it suffices to show that $\mathcal {M}_p\not=\Rset^2$, i.e., the existence of gaps between periodic orbits (Modulo $\Zset^2$). Nevertheless, it is difficult (if ever possible) to verify this for any specific $p$. The existence of at least two line segments (oppsite to each other) is proved by contradiction.  Suppose that $F_c$ does not contain line segment, then $\mathcal {M}_p=\Rset^2$ for  all $p\in F_c$ with $n_p\in \Rset\Zset^2$. This implies that every geodesics associated with the $\Zset^2$-periodic Riemannina metric on $\Rset^2$
$$
g(x)={2(c-V(x))}\left((dx_1)^2+(dx_2)^2\right)
$$ 
is a minimal geodesics. According to a classical result of Hopf \cite{H}, the metric must be flat. Consequently,  $V\equiv 0$, which contradicts our assumption.  Apparently,   this non-constructive approach guarantees only the existence of line segments in $F_c$ without providing any precise information about their numbers or locations.
\medskip

{\bf Step III:} Given $p\in F_c$, if $n_p\notin \Rset\Zset^2$, then it can be shown that any two orbits on $\mathcal{M}_p$ will approach each other as time $t\to \infty$. Hence $\mathcal {M}_p$ is connected. Therefore,  if there is a line segment on $\partial F_c$ with a normal vector $\notin \Rset\Zset^2$, Lemma \ref{lem:equal-aubry} and (\ref{eq:uniquenessset}) will lead to two distinct points $p, p'$ in the interior of the line segment such that 
$$
p\cdot x+v_p(x)=p'\cdot x+v_{p'}(x)+C
$$
for a fixed constant $C$ and all $x\in \Rset^2$, which is absurd.

\end{proof}

\medskip

\nit {\bf $\bullet$ Sketch of the Proof of Theorem \ref{main}.} 

\medskip

{\bf 1:} For \( p \in \partial F_0 \), the projected Mather set \( \mathcal{M}_p \) may consist only of points in \( \mathbb{Z}^2 \). Therefore, we work instead with orbits in the projected Aubry set \( \mathcal{A}_p \). Such orbits may approach $\Zset^2$ at one or two ends. To address that, our strategy is to use extension and unit-speed reparametrization to appropriately connect orbits, forming globally defined orbits that are unbounded in both time directions \( t \to \pm \infty \). These orbits are referred to as \emph{ECCUS}; see Definition~\ref{def:cus}.

{\bf One essential difference:} Unlike the case \(c > 0\), due to the presence of critical points, two distinct ECCUS on \(\mathcal{A}_p\) for \(p \in \partial F_0\) may intersect. Furthermore, two distinct ECCUS may intersect at more than one point.

\medskip

{\bf 2:} A suitable version of previous {\bf Step~I} can be established using  periodic (modulo \( \mathbb{Z}^2 \)) ECCUS orbits in \( \mathcal{A}_p \). More precisely, if there exists a gap between two such orbits, we can show that \( p \) is a linear point by applying an argument similar to the simplified proof given after {\bf Step~I}, which relies on the upper semicontinuity of \( \mathcal{A}_p \) in two dimensions. We would like to emphasize that upper semicontinuity is a useful tool for simplifying the proof, but it is not essential.

\medskip

\textbf{3 (Key Step):} This step contains the main difficulty of the proof, relying on a crucial result(Theorem~\ref{theo:no-flat-fill}), which characterizes the limiting behavior of action-minimizing curves near critical points. 

Suppose that $\partial F_0$ is differentiable at a point $p$ with the unit outer normal vector $q_p\in \mathbb{R}\mathbb{Z}^2$.  If there does not exist a gap bounded by two  periodic (modulo \( \mathbb{Z}^2 \)) ECCUS orbits in \( \mathcal{A}_p \), there must exist a sequence of  periodic (modulo $\mathbb{Z}^2$) orbits in the projected Mather set $\mathcal{M}_p$ that approaches the origin $(0,0)$. Theorem~\ref{theo:no-flat-fill} then implies that
$$
\gamma_{b0}(t), \quad \gamma_{b0}(-t), \quad \gamma_{b1}(t), \quad \gamma_{b1}(-t)
$$
must be homoclinic orbits contained in $\mathcal {A}_p\cup \mathcal{A}_{-p}$ (Corollary \ref{cor:uniqueaubry}).  Here for $i=1,2$, $\gamma_{bi}:\Rset\to \Bbb T^2$ is the unique stable orbit (if exists) that is tangent to the eigenvector associated with the larger eigenvalue $b^2$, i.e.,
\be\label{homo-aubry}
\begin{cases}
\ddot \gamma_{bi}(t)=-DV(\gamma_{bi}(t))\\[3mm]
\lim_{t\to \infty}\gamma_{bi}(t)=(0,0)\\[3mm]
\lim_{t\to \infty}{\gamma_{bi}(t)\over |\gamma_{bi}(t)|}=(-1)^iv_b.
\end{cases}
\ee
Recall that $|v_b|=1$ and $D^2V((0,0))v_b=b^2v_b$. The existence and uniqueness of such a curve \( \gamma_{bi} \)  are well known to experts.  For reader's convenience, we present the proof in Lemma~\ref{lem:vunique}.

\medskip

\subsection{Homoclinic orbits}

\begin{defin} For $p\in \partial F_0$, a curve $\gamma:\Rset\to \Rset^2$ is called a homoclinic orbit associated with $p$ if 
\be\label{two-ends}
\lim_{t\to +\infty} \gamma(t)=(m,n) \quad \mathrm{and} \quad \lim_{t\to -\infty} \gamma(t)=(0,0)
\ee
for some $(m,n)\in \Zset^2\backslash \{0\}$ and
$$
\int_{-\infty}^{\infty}{1\over 2}|\dot \gamma|^2-V(\gamma(t))\,dt=p\cdot (m,n).
$$
\end{defin}
It is obvious that  $(m,n)$ represents the first homology class of $\gamma$ when projected onto $\mathbb{T}^2$.  
\begin{defin}
Denote by
\[
G_p = \{\text{first homology classes of homoclinic orbits associated with } p\}.
\]
\end{defin}

Clearly, each homoclinic curve is an orbit on the projected Aubry set $\mathcal{A}_p$.  See \cite{Cheng} for more detailed information related to homoclinic curves.  For reader's convenience, we present some relevant results and proofs here.

\begin{lem}\label{lem:G-property} Assume that $p\in \partial F_0$. 

(1) If $(m,n)\in G_p$, then it is primitive, i.e., $(m,n)=(\pm 1,0)$ or $(m,n)=(0, \pm 1)$ or $\mathrm{gcd}(|m|,|n|)=1$.
\medskip

(2) If $(m,n)\in G_p$, then its normalization
$$
{(m,n)\over \sqrt{m^2+n^2}}\in n_p.
$$

(3) $G_p$ contains at most three elements. Also, if $G_p$ has two or three elements, then 
$$
G_p=\{v_0, v_1\}  \quad \mathrm{or} \quad G_p=\{v_1, v_2, v_1+v_2\} 
$$
for a positively oriented unimodular pair $v_0, v_1\in \Zset^2$, i.e.,
\be\label{eq:unimodular}
\det A=1.
\ee
Here $A=(v_1^{\intercal}\,\, v_0^{\intercal})$ is the $2\times 2$ matrix whose columns are  $v_0^{\intercal}$ and $v_1^{\intercal}$.

\end{lem}

\nit Proof: (1)  follows from the fact that a homoclinic orbit \( \gamma \) associated with \( p \) cannot intersect itself when projected onto the flat torus \( \mathbb{T}^2 \). Equivalently, for any \( \vec{w} \in \mathbb{Z}^2 \setminus \{0\} \), the curves \( \gamma \) and \( \gamma + \vec{w} \) in \( \mathbb{R}^2 \) are either disjoint or coincide up to a time translation.

(2) Note that for any $p'\in \partial F_0$ and any $\Zset^2$-periodic viscosity subsolution $v'$ of the associated problem~\eqref{mech-cell}, we have
\[
\int_{-\infty}^{\infty}\left(\tfrac{1}{2}|\dot \gamma(t)|^2 - V(\gamma(t))\right)\,dt 
\;\geq\; p'\cdot (m,n) + v'(\gamma(\infty)) - v'(\gamma(-\infty)) 
\;=\; p'\cdot (m,n).
\]
Hence, for all $p'\in  \partial F_0$,
\[
(p-p')\cdot (m,n) \;\geq\; 0.
\]

(3) Suppose that for $i=0,1,2$, the vectors $(m_i,n_i)$ are three distinct elements of $G_p$, and let $\gamma_i$ be the associated homoclinic orbits.  
By (2), without loss of generality we may assume that
\[
(m_2,n_2)=\lambda_2 (m_0,n_0)+\lambda_{2}'(m_1,n_1),
\qquad \lambda_2,\lambda_{2}'>0.
\]

Using $\gamma_1$, $\gamma_2$, and their integer translates, one can form a curved lattice net such that no integer point lies in the interior of any cell. See \ref{f2} below. 

By planar topology and the fact that the curves $\gamma_i$ are mutually disjoint, it follows that
\[
(m_2,n_2)=(m_0,n_0)+(m_1,n_1).
\]
Finally, $|\mathrm{det}(A)|$ represents the minimal possible intersection number of two simple closed curves on $\mathbb{T}^2$ with given first homology class $v_0$ and $v_1$ when both of them are primitie integer vectors. See section 1.2.3 in \cite{FM2012} for instance. 
\begin{center}
\includegraphics[scale=0.5]{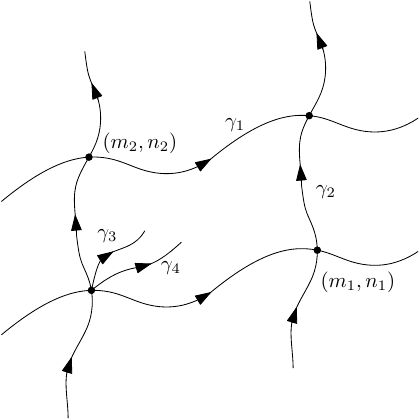}
\captionof{figure}{}
 \label{f2}
\end{center}
 \qed

\begin{rmk} It is possible that $G_p$ contains three elements. For instance, in Example \ref{example:separable} in section 6, for $p=(L_1,L_2)$,
$$
G_p=\{(1,0),\ (0,1), \ (1,1)\}.
$$

\end{rmk}

\noindent
\textbf{$\bullet$ Periodic Extension of Homoclinic Orbits.} Since a homoclinic orbit is bounded in $\mathbb{R}^2$, we need to extend it for our purposes. Suppose $\gamma:\mathbb{R} \to \mathbb{R}^2$ is a homoclinic curve associated with $p$. We define its periodic extension $\tilde{\gamma}$ by the following reparametrization (see Figure \ref{f3}):
\begin{equation} \label{eq:extension}
\tilde{\gamma}(t) =
\begin{cases}
\gamma \left( \arctan\left( \pi \left(t - \tfrac{1}{2} \right) \right) \right) & \text{for } t \in [0,1], \\
\tilde{\gamma}(t - k) + k(m,n) & \text{for } k \in \mathbb{Z},\ t \in [k, k+1].
\end{cases}
\end{equation}
For convenience, we may also consider a unit-speed reparametrization of $\tilde{\gamma}$.

Given a viscosity subsolution $v$ of~\eqref{mech-cell}, let $u = p \cdot x + v$. Then for any $s_1 < s_2$,
\[
u(\tilde{\gamma}(s_2)) - u(\tilde{\gamma}(s_1)) = h(\tilde{\gamma}(s_1), \tilde{\gamma}(s_2)).
\]

In the next section, we extend this construction to obtain globally defined curves with the above property by suitably connecting other orbits contained in $\mathcal{A}_p$.

\begin{center}
\includegraphics[scale=0.5]{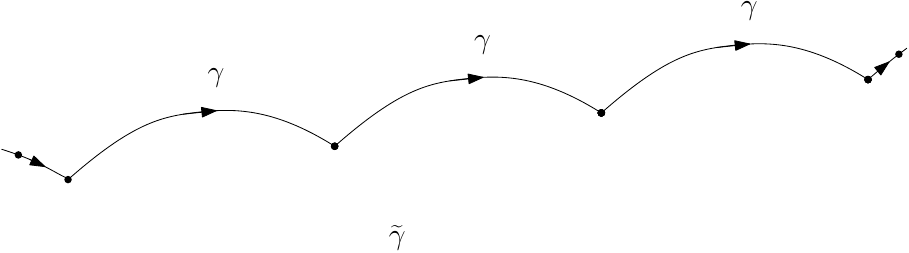}
\captionof{figure}{}
 \label{f3}
\end{center}

\subsection{Extension and Unit-reparameterization of Orbits}

 Due to the presence of critical points, orbits in $\mathcal{A}_p$ may converge to $\mathbb{Z}^2$ at one end or at both ends (as in the case of homoclinic orbits), which implies that these orbits are bounded in at least one direction. For our purposes, we need to connect such orbits in a suitable way in order to construct unbounded ones along both directions ($t\to\pm \infty$).

\begin{defin}\label{def:cus} 
Let $p\in \Rset^n$ and let $v$ be a viscosity subsolution of~\eqref{mech-cell}.  
Define $u = p \cdot x + v$.  
A Lipschitz continuous curve $\xi:\Rset \to \Rset^2$ is called an \emph{extended calibrated curve with unit speed} (ECCUS) associated with $(p,v)$ if the following hold:

\begin{enumerate}
\item For all $s_1 < s_2$, 
\[
u(\xi(s_2)) - u(\xi(s_1)) = h(\xi(s_1),\xi(s_2)).
\]

\item 
\[
|\dot \xi(s)| = 1 \qquad \text{for a.e. } s \in \Rset.
\]
\end{enumerate}
\end{defin}

For convenience, we often omit the dependence of an ECCUS on $v$. If $\xi: \mathbb{R} \to \mathbb{R}^n$ is an ECCUS, it is expected to be the union of unit speed reparametrizations of distinct global characteristics. Two notable examples are:
\begin{enumerate}
    \item the unit-speed reparametrization of an unbounded (along both directions) orbit on $\mathcal{A}_p$, and
    \item the unit-speed reparametrization of the periodic extension of a homoclinic orbit, as defined in~\eqref{eq:extension}.
\end{enumerate}

For reader's convenience, we presented the proof below (Lemma \ref{lem:cusproperty}).  Set
$$
Z_\xi=\{s\in \Rset|\ \xi(s)\in \Zset^2\}.
$$
Apparently,  $|s_1-s_2|\geq 1$ for $s_1\not=s_2\in Z_{\xi}$. We write 
$$
\Rset\backslash Z_\xi=\cup_{j\in J}I_j.
$$
Here $J$ is a finite or countably many set. The collection  $\{I_j\}_{j\in J}$ are disjoint open sets.

\begin{lem}\label{lem:cusproperty} Given $p\in \Rset^n$ and  a viscosity subsolution $v$ of~\eqref{mech-cell},  let $\xi:\Rset\to \Rset^n$ be an ECCUS  associated with ($p$,$v$). Then

\medskip

(1) $\xi:\Rset\to \Rset^2$ is one-to-one and  for all $s_1<s_2<s_3$
\be\label{eq:custriangle}
h(\xi(s_1),\xi(s_3))=h(\xi(s_1),\xi(s_2))+h(\xi(s_2),\xi(s_3)).
\ee

(II) $\xi\in C^{\infty}(\Rset\backslash Z_{\xi})$ is a unit speed reparametrization of a global characteristic $\eta_j$ of $v$ within each $I_j$ for all $j\in J$. In particular, if $I_j$ is finite, that $\eta_j$ is a homoclinic orbit associated with $p$;

(III) For any $s_1,s_2\in \Rset$ and $\delta_1$ from Corollary \ref{cor:lower-gap},
$$
u(\xi(s_2))-u(\xi(s_1))\geq \delta_1(|s_2-s_1|-1).
$$
\end{lem}

Proof:  (I) Given $t_1<t_2$, clearly, there exists $t\in (t_1,t_2)$ such that $\xi(t)\not=\xi(t_1)$. Then
$$
\begin{array}{ll}
u( \xi(t_2))-u(\xi(t_1))&=u( \xi(t_2))-u(\xi(t))+u(\xi(t))-u(\xi(t_1)\\[3mm]
&\geq u(\xi(t))-u(\xi(t_1))=h(\xi(t_1),\xi(t))>0,
\end{array}
$$
i.e., $u$ is strictly increasing along $\xi$. Hence $\xi$ is one to one. (\ref{eq:custriangle}) follows immediately from the definition. 
\medskip

(II) Fix $j\in J$. It suffices to show that for  any  $t_0\in I_j$, there exists $\epsilon_0>0$, such that $\xi:(t_0-\epsilon_0,t_0+\epsilon_0)\to \Rset^n$ is the unit reparametrization of a characteristics of $v$. 

In fact, Choose $\epsilon_0$ such that for  $t\in [t_0-\epsilon_0, t_0+\epsilon_0]\subset I_j$,
$$
|\xi(t)-\xi(t_0)|<r_{\xi(t_0)}
$$
Here $r_{\xi(t_0)}$ is from (3) of Lemma \ref{h-bound}. Choose $a,b\in (t_0-\epsilon_0, t_0+\epsilon_0) $ such that $a<t_0<b$.  there exist two curves $\eta_{-}:[\alpha_,\alpha_0]\to \Rset^2$ and $\eta_{+}:[\alpha_0,\alpha_+]\to \Rset^2$ such that $\eta_-(\alpha_-)=\xi(a)$, $\eta_-(\alpha_0)=\xi(t_0)=\eta_+(\alpha_0)$ and
$\eta_+(\alpha_+)=\xi(b)$,
$$
h(\xi(a),\xi(t_0))=\int_{\alpha_-}^{\alpha_0}\left({1\over 2}|\dot \eta_-|^2-V(\eta_-)+\overline H(p)\right)\,ds,
$$
$$
h(\xi(t_0),\xi(b))=\int_{\alpha_0}^{\alpha_+}\left({1\over 2}|\dot \eta_+|^2-V(\eta_+)+\overline H(p)\right)\,ds.
$$
Hence
$$
u(\xi(b))-u(\xi(a))=\int_{\alpha_-}^{\alpha_+}\left({1\over 2}|\dot \eta|^2-V(\eta)+\overline H(p)\right)\,ds.
$$
Here
$$
\eta(t)=
\begin{cases}
\eta_-(t) \quad \text{for $t\in [\alpha_-,\alpha_0]$}\\[3mm]
\eta_+(t) \quad \text{for $t\in [\alpha_0,\alpha_+]$}.
\end{cases}
$$
Therefore $\eta$ is a characteristics of $v$ and $v$ is differentiable on $\eta((a, b))$ (in particular, $\eta(0)=\xi(t_0)$).  Consequently,  $\eta$ is the unique solution to $\ddot \eta=-DV(\eta)$ subject to $\eta(\alpha_0)=\xi(t_0)$ and $\dot \eta(\alpha_0)=Du(\eta(\alpha_0))$. Note that $\xi(a),\xi(b)\in \eta(\Rset)$. Since $a$ and $b$ are arbitrary, $\xi$ is a reparametrization of $\eta$ and our conclusion holds.

\medskip

(III) follows immediately from (II) and Corollary \ref{cor:lower-gap}.

\qed

The following is an extension of a similar result for global characterstics when $\overline H(p)>0$. See \cite{W-E, G2002} for instance.
\begin{lem}\label{lem:direction} Let  $p\in \partial F_0$ and $\xi$ be an ECCUS  associated with $(p,v)$. Then

1. For any subsequence $s_m\to \infty$ as $m\to \infty$,  if 
$$
\lim_{m\to \infty}{\xi(s_m)\over |\xi(s_m)|}=q,
$$
then $q\in n_p$;

2.  If there are  $s_1<s_2$ such that $\xi(s_2)-\xi(s_1)=w\in \Zset^2$ for $i=1,2$ and, then
$$
{w\over |w|}\in n_p.
$$
\end{lem}

Proof: (1) For any $p'\in \partial F_0$ and an associated viscosity solution $v_{p'}$ to (\ref{mech-cell}), we have that 
$$
\begin{array}{ll}
&h(\xi(0),\xi(s_m))=p\cdot \xi(s_m)-p\cdot\xi(0)+v_{p}(\xi(s_m))-v_{p}(\xi(0))\\[3mm]
&h(\xi(0),\xi(s_m))\geq p'\cdot \xi(s_m)-p'\cdot \xi(0)+v_{p'}(\xi(s_m))-v_{p'}(\xi(0))
\end{array}
$$
So
$$
p\cdot (\xi(s_m)-\xi(0))+v_{p}(\xi(s_m))-v_{p}(\xi(0))\geq p'\cdot (\xi(s_m)-\xi(0))+v_{p'}(\xi(s_m))-v_{p'}(\xi(0)).
$$
Owing to (3) in Lemma \ref{lem:cusproperty}, $\lim_{|t|\to \infty}|\xi(t)|=\infty$.  Dividing $|\xi(s_m)|$ on both sides and sending $m\to \infty$, we deduce that
$$
p'\cdot q\leq p\cdot q.
$$
Since this holds for all $p'\in \partial F_0$,   (1)  holds.

(2) Adopting  the same notations from the above proof,  
we have that 
$$
\begin{array}{ll}
&h(\xi(s_1),\xi(s_2))=p\cdot w\\[3mm]
&h(\xi(s_1),\xi(s_2))\geq p'\cdot w.
\end{array}
$$
Thus for $q={w\over |w|}$,
$$
p'\cdot q\leq p\cdot q \quad \text{for all $p'\in \partial F_0$}.
$$
Hence (2)  holds. 

\qed

The following two properties can be readily verified. We leave their proofs to the reader as an exercise.

\begin{lem}\label{lem:cusaubry}
If $\xi:\Rset\to \Rset^n$ is a periodic (modulo $\Zset^n$) ECCUS associated with $(p,v)$, then
$$
\xi(\Rset)\subset \mathcal{A}_p.
$$
\end{lem}

\begin{lem}\label{lem:cusstability}
For $m\geq 1$,  let $\xi_m$ be an  ECCUS associated with $(p_m,v_m)$.  If $\lim_{m\to \infty}p_m=p$, $\lim_{m\to \infty}v_m=v$ and $\lim_{m\to \infty}\xi_m(s)=\xi(s)$ for all $s\in \Rset$, then $\xi$ is an ECCUS associated with $(p,v)$. 
\end{lem}

Next we prove the following Corollary.

\begin{cor}\label{cor:aubrynonempty} For any $p\in \partial F_c$, there exists an periodic ECCUS $\xi$ such that
$$
\xi(\Rset)\subset \mathcal{A}_p.
$$
In particular, when $n = 2$, if there exists a point $q_p\in n_p\cap \Rset\Zset^2$, then there exists an ECCUS $\xi$ whose first homology class is $(m, k) \in \mathbb{Z}^2$. Here $(m,k)$ is a primitive integer vector that is on the same direction as $q_p$.

\end{cor}

Proof:  {\bf Case 1:} Assume that the cell problem ~\eqref{mech-cell} has a unique solution up to a constant. Then any global characteristic lies entirely on \(\mathcal{A}_p\). Therefore, it suffices to establish the existence of an ECCUS.

To this end, choose a sequence \(\{p_m\}_{m \geq 1}\) such that \(p_m \to p\) and \(\overline{H}(p_m) = \frac{1}{m}\). Let \(\xi_m\) be an orbit on \(\mathcal{A}_{p_m}\); then \(\xi_m\) is unbounded ($\lim_{|t|\to \infty}|\xi_m(t)|=\infty$).  With a slight abuse of notation, we denote by \(\xi_m\) its unit-speed reparametrization.

Passing to a subsequence if necessary, we may assume that \(\xi_m \to \xi\), where \(\xi\) is an ECCUS by the stability lemma~\ref{lem:cusstability}. Here for convenience, we omit the dependence of $\xi$ on viscosity subsolutions to (\ref{mech-cell}).

\medskip

{\bf Case 2:}  Assume that the cell problem~\eqref{mech-cell} admits more than one solution up to an additive constant.  
By Corollary~\ref{cor:lower-gap}, for any orbit $\xi$ in $\mathcal{M}_p$, either $\lim_{t\to\infty}\xi(t)=(0,0)$ or $\lim_{t\to\infty}|\xi(t)|=\infty$ (and similarly as $t\to -\infty$).  
Since $\mathcal{M}_p$ is the uniqueness set for solutions of the cell problem~\eqref{mech-cell}, it must contain an orbit $\xi$ such that $\lim_{|t|\to\infty}|\xi(t)|=\infty$.  
Consequently, a unit-speed reparametrization of $\xi$ is an ECCUS.

Now let us assume that $n=2$ and $q_p=\lambda (m,k)$ for a primitive integer vector $(m,k)$. For $l\in \Nset$,  we may choose $p_l\to p$ such that $\overline H(p_l)={1\over l}$ and $q_p$ is the outward unit normal vector at $p_l$ of the set $F_{1\over l}$. Let $\xi_{l}:\Rset\to \Rset^2$ be a periodic (modulo  $\Zset^2$) orbit on $\mathcal{M}_{p_l}$. By abusing the notation, denote by $\xi_{l}$ its unit reparametrization and $T_{l}$ be its minimum period. Then $\xi_{l}(T_{l})-\xi_{l}(0)=(m,k)$.  Owing to Corollary \ref{cor:lower-gap}, we have that
$$
\sqrt{m^2+k^2}\leq T_{l}\leq {1\over \delta_1}\sqrt{m^2+k^2}\sqrt{2(1-\min_{\Rset^2}V)}+1. 
$$
 Passing to a subsequence if necessary, we may assume that as $l \to \infty$,  
$T_l \to T_\infty$ and $\xi_l \to \xi_\infty$. By stability,  $\xi_\infty$ is a periodic (modulo $\Zset^2$) ECCUS satisfying  
\[
\xi_\infty(T_\infty) - \xi_\infty(0) = (m,k).
\]

\qed

\begin{lem}\label{lem:twopointintersection}
Given \( p_1, p_2 \in \partial F_0 \), suppose two distinct ECCUSs \( \xi_1 \) and \( \xi_2 \), associated with $(p_1,v_1)$ and $(p_2,v_2)$ respectively, intersect at two distinct points. Then there exist two distinct integer points in \( \xi_1(\mathbb{R}) \cap \xi_2(\mathbb{R}) \). In particular, this implies that \( n_{p_1} \cap n_{p_2}\cap \Rset\Zset^2 \neq \emptyset \) or \( n_{p_1} \cap n_{-p_2} \cap \Rset\Zset^2\neq \emptyset \) .

\end{lem}

Proof:  Note that for $i=1,2$,  $u_i=p_i\cdot x+v_i$ is strictly increasing along $\xi_i$. 

Since $\xi_1$ and $\xi_2$ are distinct, after a time translation if necessary, there exist $s_1,s_2>0$ such that one of the following holds:

\medskip
\noindent
\textbf{Case 1:}
\[
\xi_1(0)=\xi_2(0), \quad 
\xi_1(s_1)=\xi_2(s_2), \quad 
\xi_1((0,s_1)) \cap \xi_2((0,s_2)) = \emptyset.
\]

\medskip
\noindent
\textbf{Case 2:}
\[
\xi_1(0)=\xi_2(s_2), \quad 
\xi_1(s_1)=\xi_2(0), \quad 
\xi_1((0,s_1)) \cap \xi_2((0,s_2)) = \emptyset.
\]

Let us first focus on Case 1.

\medskip

{\bf $\bullet$ Gluing of two curves.} Define the following curve (see Figure \ref{f4} below)
$$
\widetilde \xi(s)=
\begin{cases}
\xi_1(s) \quad \text{for $t\leq 0$}\\[3mm]
\xi_2(s) \quad \text{for $0\leq t\leq s_2$}\\[3mm]
\xi_1(s-s_2+s_1) \quad \text{for $t\leq s_2$.}
\end{cases}
\medskip
$$

\begin{center}
\includegraphics[scale=0.7]{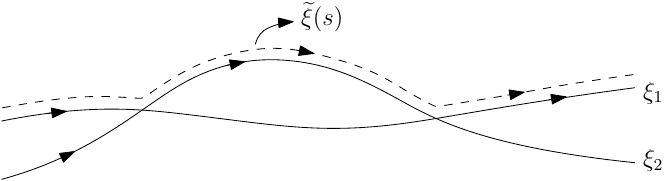}
\captionof{figure}{}
 \label{f4}
\end{center}

We claim that $\tilde \xi(t)$ is also an ECCUS associated with $(p_1,v_1)$. We will show that  all $0\leq \alpha<\beta\leq s_2$,
$$
u_1(\xi_2(\beta))-u_1(\xi_2(\alpha))=h(\xi_2(\alpha),\xi_2(\beta)).
$$
The argument for other situations are similar.  In fact,
$$
\begin{array}{ll}
u_1(\xi_2(s_2))-u_1(\xi_2(0))=u_1(\xi_1(s_1))-u_1(\xi_1(0))&=h(\xi_1(s_1),\xi_1(0))\\[3mm]
&=h(\xi_2(s_2),\xi_2(0)).
\end{array}
$$
Meanwhile,  for $t_0=0$, $t_1=\alpha$, $t_2=\beta$ and $t_3=s_2$, 
$$
u_1(\xi_2(s_2))-u_1(\xi_2(0))=\sum_{i=0}^{2}(u_1(\xi_2(t_{i+1}))-u_1(\xi_2(t_i))).
$$
Due to (\ref{eq:custriangle}),
$$
h(\xi_2(s_2))-h(\xi_2(0))=\sum_{i=0}^{2}(h(\xi_2(t_{i+1}))-h(\xi_2(t_i))).
$$
Combining with $u_1(y)-u_1(x)\leq h(x,y)$, our claim holds.

If $\xi_1(0)\not\in \Zset^n$, then both $\xi_1$ and $\widetilde \xi$ are unit reparametrization of a characteristics of $v_1$ near $t=0$. Accordingly, they must coincide  near 0 after a suitable time transformation, which contradicts the assumption in Case 1. Similarly, we deduce that $\xi_1(s_1)\in \Zset^n$. Then due to (2) in Lemma \ref{lem:direction}, $n_{p_1}\cap n_{p_2}\not=\emptyset$.

As for Case 2, the proof is similar by noticing that $\xi_2(s_2-t)$ is an ECCUS associated with $(-p_2,-v_2)$.  It leads to 
$n_{p_1}\cap n_{-p_2}\cap \Rset\Zset^2\not=\emptyset$
\qed

\section{Proof of Theorem \ref{main}}\label{sec:main}

We first associate each ECCUS with something similar to a circle map, thereby extending the analogous construction used for \( F_c \) with \( c > 0 \) (see \cite{B1}). Recall that an ECCUS is defined in Definition~\ref{def:cus}.

\medskip

{\bf $\bullet$ Curved coordinate system.} Given a primitive integer vector $(m,n)$ and a periodic (modulo $\Zset^2$) ECCUS $\eta$ associated with $(\widetilde{p},\widetilde{v})$, for $k \in \Zset$ define
\be\label{eq:curvecoordinate}
\eta_k(t) = \eta(t) + k(m,n), \qquad t \in \Rset.
\ee
Assume that for any two distinct integers $k_1$ and $k_2$, 
\[
\eta_{k_1}(\Rset) \cap \eta_{k_2}(\Rset) = \emptyset.
\]
Now let $\xi$ be an ECCUS associated with $(p',v')$, and suppose that $\xi$ intersects each $\eta_k$ exactly once for all $k \in \Zset$. 
For each $k \in \Zset$, choose the unique $t_k \in \Rset$ such that (see Figure \ref{f5})
\[
\eta_k(t_k) \in \xi(\Rset).
\]
Define $g_\xi:\Zset\to \Rset$ as
\be\label{eq:gfunction}
g_\xi(k)=t_k.
\ee

\begin{center}
\includegraphics[scale=0.3]{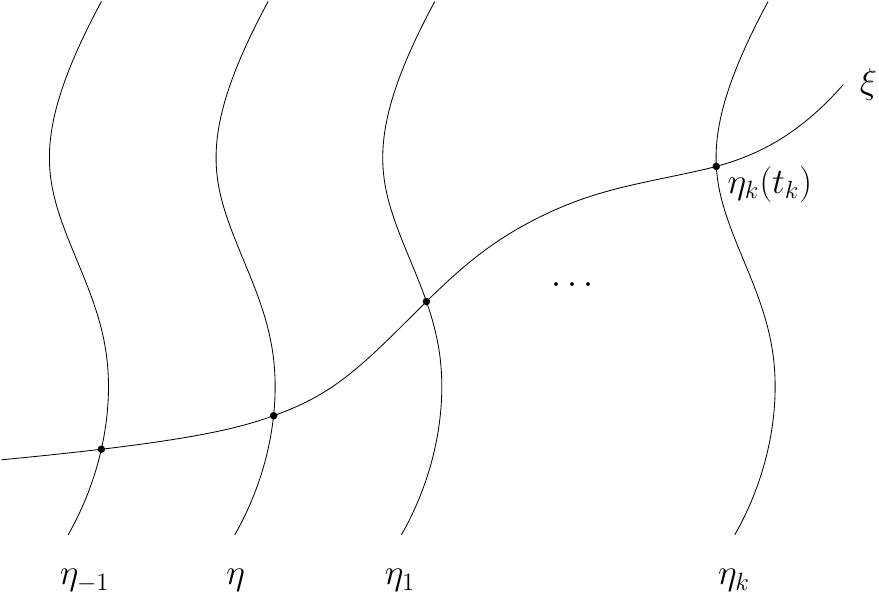}
\captionof{figure}{}
 \label{f5}
\end{center}

\begin{lem} [Monotonicity]\label{lem:mon} 
If  $g_\xi$ is not strictly monotonic, then  
$$
{(m,n)\over |(m,n)|}\in n_{p'} \quad \mathrm{or} \quad -{(m,n)\over |(m,n)|}\in n_{p'}.
$$
\end{lem}

Proof: If $g(k)$ is not strictly monotone, then $\xi$ and $\xi+(m,n)$ have intersections. Since both curves are ECCUS associated with $(p',v')$,  due to (2) in Lemma \ref {lem:cusproperty}, either $\xi(t+t_0)=\xi(t)+(m,n)$ for some fixed $t_0\in \Rset$ and all $t\in \Rset$ or 
$$
\xi(R)\cap (\xi(R)+(m,n))\cap \Zset^2\not=\emptyset. 
$$
For either case, there exits $t_1,t_2\in \Rset$ such that $\xi(t_2)-\xi(t_1)=(m,n)$.   Consequently,   our conclusion follows from (2) of Lemma \ref{lem:direction}. \qed

Suppose that $p_0\in \partial F_0$ and let 
\[
q_0=\frac{(m,n)}{|(m,n)|},
\]
where $(m,n)$ is a primitive integer vector. 

\medskip

{\bf Assumption 1:}  Either $n_{p_0}=\{q_0\}$, or else $q_0$ on the counterclockwise  end of $n_{p_0}$ (that is, 
\(\det(q_1,q_0)>0\) for every other $q_1\in n_{p_0}$ with $q_1\neq q_0$).

\medskip

{\bf Assumption 2: } No line segment lies on the counterclockwise direction:
$$
\partial F_0\cap \{p_0+tq_{0}^{\perp}|\ t>0\}=\emptyset. 
$$
Denote by $\mathcal{P}_0$ as follows 
\begin{equation}\label{eq:P_0}
\mathcal{P}_0
:= \left\{ \gamma \;\middle|\;
\begin{array}{l}
\gamma \text{ is a periodic (mod } \mathbb{Z}^2\text{) ECCU associated with } p_0,\\
\text{whose first  homology class is } (m,n)\in \mathbb{Z}^2
\end{array}
\right\}.
\end{equation}
 Thanks to Corollary~\ref{cor:aubrynonempty},  $\mathcal{P}_0\not=\emptyset$. 

\begin{defin}\label{def:gap} Assume that $\xi_0, \xi_1\in \mathcal{P}_0$ and $\xi_0(\Rset)\cap \xi_1(\Rset)=\emptyset$.  The open strip $U$ bounded by $\xi_0$ and $\xi_1$ is called a gap if  there does not exist $\xi\in \mathcal{P}_0$ satisfying $\xi(\Rset)\subset \overline{U}$ and $\xi(\Rset)\cap {U}\not=\emptyset$. 
    
\end{defin}

Note  that $\xi_0(\Rset)$ separates $\Rset^2$ into two open sets $\Omega_+$ and $\Omega_-$, and there exists $M_0>0$ such that
\be\label{eq:twodomainsep}
\{x\in\Rset^2:\ (-n,m)\cdot x\ge M_0\}\subset \Omega_+,
\
\{x\in\Rset^2:\ (-n,m)\cdot x\le -M_0\}\subset \Omega_-.
\ee
Without loss of generality, we assume the following. 
\medskip

{\bf Assumption 3:}  $\xi_1(\Rset)\subset \Omega_+$.

\medskip
Hence $U\subset \Omega_+$.  Therefore $\xi_0(\Rset)$ and $\xi_1(\Rset)$ separate $\Rset^2$ into three disjoint domains $\Omega_{++}$, $U$ and $\Omega_-$. See Figure \ref{f6} below. 
\medskip

\begin{center}
\includegraphics[scale=0.7]{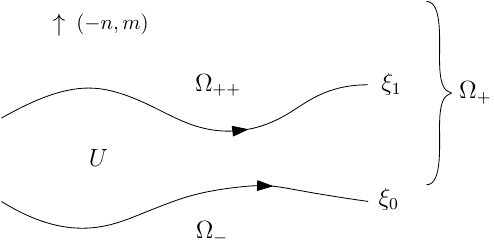}
\captionof{figure}{}
 \label{f6}
\end{center}

Below we  establish an approximation lemma.

\begin{lem}\label{lem:approxgap} Let $\xi_0, \xi_1\in \mathcal{P}_0$. Suppose the open strip $U$ bounded by $\xi_0$ and $\xi_1$ is a gap. Given Assumptions 1,2 and 3,   there exists an ECCUS $\gamma_+$   associated with $p_0$,  such that $\gamma_+(0)\in U$, $\gamma_{+}(\Rset)\in \overline U$  and
$$
\begin{array}{ll}
(a)&\lim_{s\to +\infty}d(\gamma_+(s), \xi_1([\alpha_+(s),\infty)))=0\\[3mm]
(b) &\lim_{s\to -\infty}d(\gamma_+(s), \xi_0((-\infty,\alpha_-(s)]))=0\\[3mm]
\end{array}
$$
Here, \(\alpha_\pm\)  are continuous strictly increasing functions $\Rset\to \Rset$ such that
\[
\lim_{s \to +\infty} \alpha_+(s)= +\infty, \quad \text{and} \quad \lim_{s \to -\infty} \alpha_-(s) = -\infty.
\]

Moreover, if $\partial F_0$ is differentiable at $p_0$, then $\gamma_+(\Rset)\cap U$  is unbounded.

\end{lem}

We would like to mention that, unlike the case \( c > 0 \), the curves \( \gamma_{+} \) here may partially overlap with \( \xi_0 \) or \( \xi_1 \) when maximum points are present.

Before proving the lemma, we first choose a suitable $\eta$ to establish the curved coordinate system.  

Owing to {\bf Assumption 2} and $n_{p_0}\cap n_{-p_0}=\emptyset$,  we may choose  $\widetilde p\in \partial F_0$ that is close to $p_0$ from the clockwise direction such that 
\begin{itemize}
\item (1) \be\label{eq:emptyintersection}
n_{\widetilde{p}} \cap n_{p_0} =n_{\widetilde{p}} \cap n_{-p_0}= \emptyset;
\ee
\item (2) There exists $\widetilde q={(\widetilde m, \widetilde n)\over |(\widetilde m, \widetilde n)|}\in  n_{\widetilde{p}}$ for a primitive integer vector $(\widetilde m, \widetilde n)$ satisfying 
\be\label{eq:rightdirection}
(\widetilde m,\widetilde n)\cdot (-n,m)>0. 
\ee
\end{itemize}
Note that we do not choose $(\tilde m,\tilde n)=(-n,m)$ since ${(-n,m)\over |(-n,m)|}$ might be in $n_{-p_0}$.

Thanks to Corollary~\ref{cor:aubrynonempty}, there exists a periodic (modulo $\mathbb{Z}^2$)  ECCUS  $\eta$ in $\mathcal{A}_{\widetilde{p}}$ whose first homology class is $(\widetilde m,\widetilde n)$, which will be used  to establish curve coordinate (\ref{eq:curvecoordinate}) and the corresponding $g_\xi$ defined in (\ref{eq:gfunction}). 

Owing to Lemma~\ref{lem:twopointintersection} and (\ref{eq:emptyintersection}),  $\eta$ intersects $\xi_0$ and $\xi_1$ exactly once. Hence, by (\ref{eq:rightdirection}), there exists $c_0<c_1$ such that
\be\label{eq:etabound}
\begin{array}{ll}
&\eta((-\infty, c_0))\subset \Omega_-,\quad  \eta(c_0)\in \xi_0(\Rset), \quad \eta((c_0,c_1))\subset U\\[3mm]
&\eta(c_1)\in \xi_1(\Rset), \quad \eta((c_1,\infty))\subset \Omega_{++}.
\end{array}
\ee
Also, 
\be\label{eq:twoconstantg}
g_{\xi_0}(k)=c_0 \quad \mathrm{and} \quad g_{\xi_1}(k)=c_1 \quad \text{for all $k\in \Zset$}.
\ee

\medskip

\begin{center}
\includegraphics[scale=0.7]{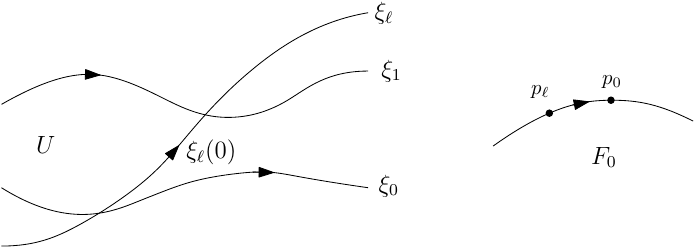}
\captionof{figure}{}
 \label{f7}
\end{center}

{\bf Proof of  Lemma \ref{lem:approxgap}}: By {\bf Assumption 1} and {\bf Assumption 2}, together with (\ref{eq:emptyintersection}),  we could choose a sequence $\{p_l\}_{l\geq 2}$ and a sequence of primitive integer vectors $\{(m_l,n_l)\}_{l\geq 2}$ such that $p_l$ approaces $p_0$ from the clockwise direction as $l\to \infty$, $q_l={(m_l,n_l)\over |(m_l,n_l)|}\in n_{p_l}$,
\be\label{eq:allempty}
\pm \frac{(m,n)}{|(m,n)|}\notin n_{p_l}  \quad \mathrm{and} \quad n_{\widetilde p}\cap n_{p_l}= n_{\widetilde p}\cap n_{-p_l}=\emptyset.
\ee
and
\be\label{eq:clockwisepositive}
\lim_{l\to \infty}q_l=q_0 \quad \mathrm{and}  \quad (m_l,n_l)\cdot (-n,m)>0
\ee
By Corollary \ref{cor:aubrynonempty},  for each $l\geq 2$, choose $\xi_l$ to be a periodic (modulo $\Zset^2$)  ECCUS on $\mathcal{A}_{p_l}$ with first homology class $(m_l,n_l)\in \Zset^2$. See Figure \ref{f7} above. Without loss of generality, we may assume that for $l\geq 2$
\be\label{eq:0condition}
|\xi_l(0)|\leq 4,\quad \xi_l(0)\in U \quad  \mathrm{and}\ d(\xi_l(0),\xi_0)={d(\xi_0,\xi_1)\over 2}.
\ee

Owing to Lemma~\ref{lem:twopointintersection} and (\ref{eq:allempty}), the curve \( \xi_l \)  intersect   \( \xi_0 \) or \( \xi_1 \) or $\eta_k$ exactly once for all $l\in \Nset$ and $k\in \Zset$. Since $(m_l,n_l)\cdot (-n,m)>0$ for each $l\geq 2$, there exist $t_{0l}<0<t_{1k}$ such that for all $l\geq 2$
\be\label{eq:xilrange}
\begin{array}{ll}
&\xi_l((-\infty, t_{0l}))\subset \Omega_-,\quad  \xi_l(t_{0l})\in \xi_0(\Rset), \quad \xi_l((t_{0l},t_{1l})\subset U\\[3mm]
&\quad \xi_l(t_{1l})\in \xi_1(\Rset), \quad \xi_l((t_{1l},\infty))\subset \Omega_{++}.
\end{array}
\ee
In addition,  thanks to Lemma \ref{lem:mon},  the function $g_l(k)=g_{\xi_l}(k):\Zset\to \Rset$  must be strictly monotonic.

Fix  $l\geq 2$. For $k\in \Zset$, let $s_{k,l}\in \Rset$ such that
$$
\xi_l(s_{k,l})=\eta_k(g_{\xi_l}(k))=\eta(g_{\xi_l}(k))+k(m,n). 
$$
Then $|s_{k,l}-s_{k+1,l}|\geq d(\eta, \eta_1)>0$. Clearly, $s_{k,l}$ is strictly monotonic with respect to $k$. Otherwise $\xi_l$ will intersect certain $\eta_i$ more than once.

\medskip

\noindent{\bf Claim 1.} $s_{k,l}$ and $g_{\xi_l}(k)$ have the same monotonicity.  

Indeed, we have either $\{s_{k,l}\}$ is strictly increasing with $\lim_{k\to\infty}s_{k,l}=\infty$, or strictly decreasing with $\lim_{k\to\infty}s_{k,l}=-\infty$.  
In the first case, when $k$ is sufficiently large, due to (\ref{eq:clockwisepositive}),  $\xi_l(s_{k,l})\in \Omega_+$ and hence $g_{\xi_l}(k)>c_1$, while for $k$ sufficiently negative, $\xi_l(s_{k,l})\in \Omega_-$ and thus $g_{\xi_l}(k)<c_0$. Therefore $g_{\xi_l}(k)$ is strictly increasing.  
The argument in the decreasing case is analogous.  
Hence the claim follows.

 Upon a subsequence if necessary, we may assume that 
$$
\lim_{l\to \infty}\xi_l=\xi_+.
$$
Then $\xi_+$ is an ECCUS associated with $(p_0,v_0)$ for some viscosity subsolution $v_0$ of (\ref{mech-cell}).  Also, by (\ref{eq:xilrange}), we have that
\be\label{eq:bothempty}
\xi_+((-\infty,0])\cap \Omega_{++}=\xi_+([0,\infty))\cap \Omega_-=\emptyset.
\ee
 In addition, $\xi_+(0)$ satisfies (\ref{eq:0condition}). 
Let 
\be\label{eq:two-t-difference}
t_+=\inf\{t>0|\ \xi_+(t)\in \xi_1(\Rset)\} \quad \mathrm{and} \quad t_-=\sup\{t<0|\ \xi_+(t)\in \xi_0(\Rset)\}.
\ee

Due to (\ref{eq:bothempty}), $\xi_+((t_-,t_+))\subset \overline U$.

\medskip

We discuss two  possible situations. See Figure \ref{f8} below. 

{\bf (I)} $t_+<+\infty$ and $t_->-\infty$. Then $\xi_+(t_+)\in \xi_1(\Rset)\cap \Zset^2$ and $\xi_+(t_-)\in \xi_0(\Rset)\cap \Zset^2$.  We could construct
$$
\gamma_+(t)=
\begin{cases}
\xi_1(t-t_++t_1) \quad \text{$t\geq  t_+$}\\
\xi_+(t) \quad \text{for $t\in [t_-,t_+]$}\\
\xi_0(t-t_-+t_0) \quad \text{ for $t\leq t_-$}.
\end{cases} 
$$
Here $\xi_1(t_1)=\xi_+(t_+)$ and $\xi_0(t_0)=\xi_+(t_-)$. Clearly, $\gamma_+$ is also an ECCUS associated with $(p_0,v_0)$.

\medskip

{\bf (II)} $t_++|t_-|=+\infty$. Without loss of generality, let us assume that $t_+=+\infty$. By (III) of Lemma (\ref{lem:cusproperty}), combining with $\gamma_+((t_-,\infty))\subset \overline U$ and (1) of Lemma \ref{lem:direction},  we deduce that 
\be\label{eq:xiplusdirection}
\lim_{t\to \infty}|\gamma_+(t)|=\infty, \quad  \lim_{t\to \infty}{\gamma_+(t)\over |\gamma_+(t)|}=q_0. 
\ee

\begin{center}
\includegraphics[scale=0.5]{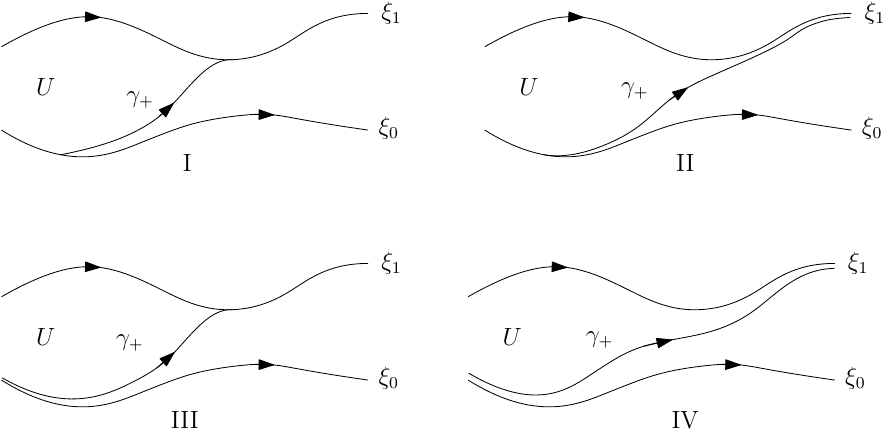}
\captionof{figure}{}
 \label{f8}
\end{center}

Also, by Lemma~\ref{lem:twopointintersection}, (\ref{eq:allempty}), and part (1) of Lemma~\ref{lem:direction}, 
$\xi_+$ intersects each $\eta_k$ exactly once. 
Hence the associated function 
\[
g_{\xi_+}(k)=\lim_{l\to\infty} g_{\xi_l}(k)
\]
is monotone with respect to $k$. For each $k\in\Zset$, let $s_k$ be defined by
\be\label{eq:xiplusintersection}
\xi_+(s_k)=\eta_k(g_{\xi_+}(k))=\eta(g_{\xi_+}(k))+k(m,n).
\ee
Then $\lim_{|k|\to \infty}|s_k|=\infty$ and  $s_k$ is strictly monotone with respect to $k$. Clearly, $s_k=\lim_{l\to\infty} s_{k,l}$. Consequently, by {\bf Claim~1}, if $\{s_k\}$ is strictly increasing (respectively, strictly decreasing), 
then the function $g_{\xi_+}(k)$ is non-decreasing (respectively, non-increasing) in $k$.

\noindent{\bf Claim 2.} $g_{\xi_+}(k):\Zset\to \Rset$ is non-decreasing.

If not, then by the above discussion we have $\lim_{k\to -\infty}s_k=\infty$. 
Since $\xi_+((t_-,\infty))\subset \overline U$, it follows that for $k$ sufficiently negative, 
$\xi_+(s_k)\in \overline U$. So $c_0 \leq g_{\xi_+}(k)\leq c_1$ for all such $k$, where $c_0$ and $c_1$ are the constants from 
(\ref{eq:etabound}) and (\ref{eq:twoconstantg}). 
Therefore, $\eta(g_{\xi_+}(k))$ is uniformly bounded for for $k$ sufficiently negative.  This contradicts (\ref{eq:xiplusdirection}) and (\ref{eq:xiplusintersection}). Hence $s_k$ is strictly increasing with respect to $k$ and our claim holds. 

If $t_->-\infty$, like (I),  we could replace $\xi_+$ by $\xi_0$ for $t\leq t_-$.  Hence, we may just assume that $\xi_+(\Rset)\in \overline U$ and take $\gamma_+=\xi_+$. Apparently, the corresponding $g_{\xi_+}$ is still non-decreasing.  Now we will establish the existence of the strictly increasing function $\alpha_+:\Rset\to \Rset$. 

 Since $g_{\xi_+}$ is non-decreasing, there are three cases:

 \medskip

 \textbf{Case 1:} Suppose \(\lim_{k \to +\infty} g_{\xi_+}(k) = c_1\). Then, as \(k \to \infty\), the segment of \(\xi_+\) between \(\eta_k\) and \(\eta_{k+1}\) must converge to the corresponding segment of \(\xi_1\) between \(\eta_k\) and \(\eta_{k+1} = \eta_k + (m,n)\). Otherwise, the approximation would yield a limiting segment \(\gamma_{\infty}\) of an ECCUS such that \(\gamma_{\infty}(0) \in U\) , $\gamma_\infty([s_1,s_2])\in \overline U$, $\gamma_\infty(s_1)$, $\gamma_\infty(s_2)\in \xi_1(\Rset)$ and
\[
\gamma_{\infty}(s_2) - \gamma_{\infty}(s_1) = (m,n)
\quad \text{for some } s_2 > 0 > s_1.
\]
The periodic extension of such a \(\gamma_\infty\) would then produce a curve in $ \mathcal{P_0}$ that  contradicts the assumption that \(U\) is a gap.

 Moreover, for any viscosity solution \(v\) of (\ref{mech-cell}), the function \(u(x) = p_0 \cdot x + v(x)\) is strictly increasing along both segments of $\xi_+$ and $\xi_1$. This implies the existence of a strictly increasing function \(\alpha_+\).

\medskip

{\bf Case 2:} Suppose \(\lim_{k \to +\infty} g_{\xi_+}(k) = c'\) for some \(c' \in (c_0, c_1)\). This leads to a contradiction, as in {\bf Case 1}: the limiting behavior would produce a curve in  $\mathcal{P}_0$ (\ref{eq:P_0})  that violates the assumption that \(U\) is a gap.

\medskip 

{\bf Case 3:}  Suppose \(g_{\xi_+}(k) = c_0\) for all \(k \in \mathbb{Z}\). Then, as before, there exist \(s_2 > 0 > s_1\) such that \(\xi_+(s_1), \xi_+(s_2) \in \xi_0(\mathbb{R})\) and \(\xi_+(s_2) - \xi_+(s_1) = (m,n)\). As in the previous cases, the periodic extension of such a segment would yield a periodic ECCUS associated with $p_0$ that contradicts the assumption that \(U\) is a gap. Therefore, this case cannot occur either.

Similarly, we can establish the existence of the strictly increasing function \(\alpha_-\).

If \(\partial F_0\) is differentiable at \(p_0\), then \textbf{Situation (I)}—where both \(t_+\) and \(t_-\) are finite—cannot occur. Indeed, by Lemma \ref{lem:direction}, the direction vector 
\[
\frac{\xi_1(t_+) - \xi_0(t_-)}{|\xi_1(t_+) - \xi_0(t_-)|}
\]
must belong to \(n_{p_0}\), which differs from \(\frac{(m,n)}{|(m,n)|}\). This implies that \(n_{p_0}\) contains more than one element, contradicting the differentiability of \(\partial F_0\) at \(p_0\). Therefore, only \textbf{Situation (II)} can occur, and consequently, \(\xi_+(\mathbb{R}) \cap U\) must be infinite.

\medskip

Finally,  $\xi_+(\Rset)\subset \mathcal{A}_{p_0}$ by the upper continuity of Aubry sets in two dimensions (See Remark \ref{rmk:upperaubry}).

 \qed

By  constructing  \(\gamma_-\) via approximating \(p_0\) counterclockwise along \(\partial F_0\), we immediately  have the following corollary.

\begin{cor}\label{cor:twodirectionapproximation}Assume that $\xi_0, \xi_1\in \mathcal{P}_0$ and $\xi_0(\Rset)\cap \xi_1(\Rset)=\emptyset$.  Suppose that  $p_0$ is a nonlinear point and $\partial F_0$ is differentiable at $p_0$. Then there exists two ECCUS associated with $p_0$, $\gamma_+$ and $\gamma_-$ on $\mathcal{A}_p$,  such that $\gamma_{\pm}(\Rset)\in \overline U$ 
$$
\begin{array}{ll}
(a)&\lim_{s\to +\infty}d(\gamma_+(s), \xi_1([\alpha_+(s),\infty)))=0\\[3mm]
(b) &\lim_{s\to -\infty}d(\gamma_+(s), \xi_0((-\infty,\alpha_-(s)]))=0\\[3mm]
(c) &\lim_{s\to +\infty}d(\gamma_-(s), \xi_0([\beta_+(s),\infty)))=0\\[3mm]
(d) &\lim_{s\to -\infty}d(\gamma_-(s), \xi_1((-\infty,\beta_-(s)]))=0\\[3mm]
\end{array}
$$
Here, \(\alpha_\pm\) and \(\beta_\pm\) are strictly increasing functions $\Rset\to \Rset$ such that
\[
\lim_{s \to +\infty} \alpha_+(s) = \lim_{s \to +\infty} \beta_+(s) = +\infty, \quad \text{and} \quad \lim_{s \to -\infty} \alpha_-(s) = \lim_{s \to -\infty} \beta_-(s) = -\infty.
\]

Moreover, if $\partial F_0$ is differentiable at $p_0$, then both $\gamma_+(\Rset)\cap U$ and $\gamma_-(\Rset)\cap U$ are unbounded.

\end{cor}

Now we are ready to prove Theorem \ref{main}.

\subsection{Proof of the ``Only if" part}.

\begin{defin}
A point \(p\in \partial F_0\) is called a \emph{special point} if it is a nonlinear point and each of the following curves
\[
\gamma_{b1}(t),\quad \gamma_{b1}(-t),\quad \gamma_{b2}(t),\quad \gamma_{b2}(-t)
\]
is a homoclinic orbit that lies  in either \(\mathcal{A}_{p}\) or \(\mathcal{A}_{-p}\).
\end{defin}

\begin{lem}\label{lem:uniquespecial}
There is at most one special point up to sign. Moreover, if a special point \(p\) exists, then \(\partial F_0\) is differentiable at \(p\), and the homology classes of \(\gamma_{b1}(t)\) and \(\gamma_{b2}(t)\) are either equal or opposite.
\end{lem}

\begin{proof}
Suppose \(p\) and \(\tilde p\) are two special points. By the definition, after possibly replacing one of them by its negative, we may assume that \(\mathcal{A}_{p}\) and \(\mathcal{A}_{\tilde p}\) contain a common homoclinic orbit. By Lemma~\ref{lem:G-property}, this implies \(n_{p}\cap n_{\tilde p}\neq\emptyset\). Hence the line segment joining \(p\) and \(\tilde p\) is contained in \(\partial F_0\), which contradicts the fact that both \(p\) and \(\tilde p\) are nonlinear points. This proves the uniqueness up to sign.

The remaining assertions follow directly from Theorem~\ref{main2} together with Lemma~\ref{lem:G-property}.
\end{proof}

Accordingly,  combining with Theorem~\ref{main2}, to prove the ``only if" part, it is sufficient to prove the following lemma. 

\begin{lem}\label{lem:onlyifpart} Give $p_0\in \partial F_0$,
suppose \(n_{p_0} = \{q_0\}\) for some \(q_0 \in \mathbb{R} \mathbb{Z}^2\). If \(p_0\) is a nonlinear point, then it has to be a special point defined above. 
\end{lem}

Proof: Assume that $q_0=\lambda_0(m,n)$ for $\lambda_0>0$ and $(m,n)$ is a primitive integer vector.  We argue by contradiction. If not, owing to Corollary \ref{cor:uniqueaubry}, there exists a gap $U$ bounded by two periodic (modulo $\Zset^2$) ECCUS  $\xi_0$ and $\xi_1$ on $\mathcal{A}_{p_0}$ with a primitive $(m,n)\in \Zset^2$ as the first homology class.

Let $\gamma_+$ and $\gamma_-$ be from Corollary \ref{cor:twodirectionapproximation}. Since $\gamma_{\pm}(\Rset)\cap U$ are unbounded,  using suitable translations along the $(m,n)$ or $-(m,n)$ directions  if needed,  we may assume that  $\gamma_+$ and $\gamma_-$ intersect in the interior of $U$.  Since both $\gamma_+$ and $\gamma_-$ are on $\mathcal{A}_{p_0}$,   at least one of the intersection points (denoted by $v_1\in U$) has to be an integer vector. Otherwise, $\gamma_+=\gamma_-$ (after suitable time translation) in $U$, which is impossible. Write $v_1=\gamma_+(s_+)=\gamma_-(s_-)\in U$. 

\medskip

{\bf Claim 1:}  $\gamma_+(\Rset)\cup \gamma_-(\Rset)\subset U$. 

\medskip

It suffices to show $\gamma_+(\Rset)\subset U$. The other one is similar.  If not, then there exists $a,b\in \Rset$ such that $\gamma_+(a)\in (\xi_0(\Rset)\cup\xi_1(\Rset))\cap \Zset^2$.   Thanks to (2) in Lemma \ref{lem:direction},  
$$
{\gamma_+(a)-v_1\over |\gamma_+(a)-v_1|}\in q_0\quad  \mathrm{or} \quad -{\gamma_+(a)-v_1\over |\gamma_+(a)-v_1|}\in q_0.
$$
Accordingly, $\gamma_+(a)-v_1$ is parallel with $(m,n)$, which implies that $v_1\in \xi_0(\Rset)\cup\xi_1(\Rset)$, contradicting the assumption that $v_1\in U$. So {\bf Claim 1} holds.

\medskip

{\bf Claim 2:}  $\gamma_+(\Rset)\cap \Zset^2=\gamma_-(\Rset)\cap \Zset^2=\gamma_+(\Rset)\cap \gamma_-(\Rset)=v_1$. 

\medskip

If not,  assume that $v_1\not=v_2\in \gamma_+(\Rset)\cap \Zset^2\subset U$. By part (2) of Lemma~\ref{lem:direction}, we must have $v_2-v_1$ is parallel with $(m,n)$. Hence, by (2) of Lemma \ref{lem:cusproperty}), $\gamma_+$ contains at least one homoclinic orbit whose first homology class is \( (m,n) \).  As a result, $U$ contain an extended homoclinic orbit whose first homology class is \( (m,n) \). This contradicts the assumption that $U$ is  a gap. Hence, {\bf Claim 2} holds.

By the two-dimensional topology, $\gamma_-$ and $\gamma_+ + (m,n)$ must also intersect in $U$.  
Similarly, this intersection point must be an integer vector $v_2 \in U$.  

\emph{Case 1.} If $v_1 = v_2$, then both $v_1$ and $v_1 - (m,n)$ lie on $\gamma_+$, 
which is impossible by the same reasoning as in the proof of Claim~2.  

\emph{Case 2.} If $v_1 \neq v_2$, then $\gamma_-$ contains two distinct integer vectors 
$v_1$ and $v_2$, which is again impossible by the same reasoning as in the proof of Claim~2.  

Thus, in either case we obtain a contradiction. \qed

\subsection{Proof of ``If" part}

Next, we prove the ``if" part.  Assume that $\partial F_0$ is differentiable at $p_0$ and $n_{p_0}=\{q_0\}$ for some $q_0\notin \Rset \Zset^2$.

Let 
$$
\mathcal{C}_{p_0}=\text{the collection of all ECCUS that belongs to $\mathcal{A}_{p_0}$}.
$$
By Corollary \ref{cor:aubrynonempty},  $\mathcal{C}_{p_0}\not=\emptyset$. In the following we will employ the method from \cite{TY} to accommodate a similar situation in which two distinct ECCUS  in $\mathcal{A}_{p_0}$ might intersect.

Choose $\widetilde{p} \in \partial F_0$ such that $n_{\widetilde{p}}\cap \Rset\Zset^2\not=\emptyset$. Select an ECCUS $\eta$ associated with $\widetilde{p}$ which is either an extension of a homoclinic orbit associated with $\widetilde{p}$, or the unit reparametrization of an unbounded periodic orbit on $\mathcal{M}_{\widetilde{p}}$.  Suppose that the first homology class of $\eta$ is $(\widetilde{m},\widetilde{n})$, which is a primitive integer vector and 
\be\label{eq:irraionalrightdrection}
(\widetilde{m},\widetilde{n})\cdot q_0>0.
\ee
If the above inequality is not true, we could just use $-\widetilde{p}$ instead. 

Without loss of generality, we assume that 
\be\label{eq:irrationalnormal}
(-\widetilde{n},\widetilde{m})\cdot q_0>0.
\ee
Employ this $\eta$ to establish the curved coordinate system with $\eta_k=\eta+k(-\widetilde{n},\widetilde{m})$ for $k\in \Zset$ if  and the corresponding $g_{\xi}$. Clearly, $\eta_{k_1}$ and $\eta_{k_2}$ are disjoint if $k_1\not=k_2$. Note if (\ref{eq:irrationalnormal}) is not true, we may just consider $\eta_k=\eta+k(\widetilde{n},-\widetilde{m})$.

Obviously, for each $\xi\in \mathcal{C}_{p_0}$, $\xi$ intersects $\eta_k$ exactly once. For $k\in \Zset$, let
\be\label{eq:irrationalcurvelabel}
\xi(s_k)=\eta_k(g_{\xi}(k))=\eta(g_{\xi}(k))+k(-\widetilde{n},\widetilde{m}).
\ee

For two curves $\xi,\xi'\in \mathcal{C}_p$, we define their distance as
$$
d(\xi,\xi')=\sum_{k\in \Zset}{\arctan (|g_{\xi}(k)-g_{\xi'}(k)|)\over |k|^2+1}
$$
Thanks to Lemma \ref{lem:twopointintersection}, $d(\xi,\xi')=0$ if and only if $\xi=\xi'$ (after suitable translation in time).

\begin{lem}\label{lem:irrationalg} Both $s_k$ and $g_\xi(k)$ are strictly increasing with $k$.
\end{lem}
Proof: Since $\xi$ cannot intersect $\eta_l$ more than once for any $l \in \mathbb{Z}$, the sequence $\{s_k\}$ must be strictly monotone in $k$. Multiplying $\pm (-\widetilde{n},\widetilde{m})$ on both sides of (\ref{eq:irrationalcurvelabel}) and combining this with (\ref{eq:irrationalnormal}), together with the fact that $|\eta \cdot (-\widetilde{n},\widetilde{m})|$ is uniformly bounded, we conclude that $s_k$ is strictly increasing in $k$ and moreover satisfies 
\[
\lim_{|k|\to \infty} |s_k| = \infty.
\]
Also, due to Lemma \ref{lem:mon}, $g_\xi(k)$ is strictly monotonic with respect to $k$. Multiplying $\pm (\widetilde{m}, \widetilde{n})$ on both sides of (\ref{eq:irrationalcurvelabel}),  (\ref{eq:irraionalrightdrection}) leads to $g_\xi(k)$ must be stricly increasing in $k$ and
\[
\lim_{|k|\to \infty} |g_\xi(k)| = \infty.
\]
\qed

The following conclusion follows immediately from the stability of ECCUS.

\begin{rmk}\label{set-closed} Given a sequence of curves $\{\gamma_l\}_{l\geq 1}\subset \mathcal{C}_p$, if 
$$
\lim_{l\to +\infty}\gamma_l=\gamma \quad \text{locally uniformly in $\Rset$},
$$
then $\gamma\in \mathcal{C}_{p_0}$ and for each $k\in \Zset$, $\lim_{l\to +\infty}g_{\gamma_l}(k)=g_\gamma(k)$.
\end{rmk}

Let $T_0>0$ be the minimum positive period of $\eta$.  Now we collect  intersections of $\xi_0$ with all curves in $\mathcal{C}_{p_0}$
$$
\mathcal{I}_{p_0}=\{t\in \Rset|\  \text{there exists $\xi\in \mathcal{C}_{p_0}$ such that $g_\xi(0)=t$}\}.
$$
Clearly,  $\mathcal{I}_{p_0}+T_0=\mathcal{I}_{p_0}$.  Owing to Remark \ref{set-closed},  $\mathcal{I}_{p_0}$ is a closed set.  Suppose that 
$$
\Rset\backslash \mathcal{I}_{p_0}=\cup_{i=1}^{\infty}(a_i,b_i),
$$
where $\{(a_i,b_i)\}_{i\geq 1}$ are disjoint open intervals. 

Now fixed $i\in \Nset$. Owing to Remark \ref{set-closed}, we could choose $\xi_{ai}$ and $\xi_{bi}$ in $\mathcal{C}_{p_0}$ such that (see the left picture on Figure \ref{f9} below)

\medskip

(1)$g_{\xi_{ai}}(0)=a_i$ and $g_{\xi_{bi}}(0)=b_i$ ; 

\medskip

(2) $d(\xi_{ai},\xi_{bi})$ attains the minimum value among all curves in $\mathcal {C}_{p_0}$ that satisfy the above (1). 

\medskip

\begin{lem} Given $i\in \Nset$,   for all  $k\in \Zset$, 
$$
g_{\xi_{ai}}(k)\leq g_{\xi_{bi}}(k).
$$
\end{lem}
Proof:  It suffices to establish this for $k>0$. The proof for $k<0$ is similar.  We argue by contradiction. If not,  let $k_0>0$ be the first positive integer such that
\be\label{eq:onesidecross}
g_{\xi_{ai}}(k)> g_{\xi_{bi}}(k).
\ee
Then $\xi_i$ and $\xi_{i}^{'}$ have to  intersect at some point between $\eta_{k_0-1}$ and $\eta_{k_0}$.  Suppose that $v_0=\xi_{ai}(t_0)=\xi_{bi}(t_{0}^{'})$. Define
\be\label{eq:glueirrational}
\tilde \xi_i(t)=
\begin{cases}
\xi_{ai}(t) \quad \text{for $t\leq t_0$}\\[3mm]
\xi_{bi}(t-t_0+t_0^{'}) \quad \text{for $t\geq t_0$}.
\end{cases}
\ee
See the middle picture on Figure \ref{f9} below. Then $\tilde \xi_i\in \mathcal{C}_{p_0}$ and $g_{\tilde \xi_i}(0)=a_i$. By (\ref{eq:onesidecross}),  we have that  
$$
d(\tilde \xi_i, \xi_{bi})\leq d(\xi_{ai}, \xi_{bi})-{\arctan (|g_{\xi_{ai}}(k_0)-g_{\xi_{bi}}(k_0)|)\over |k_0|^2+1}<d(\xi_{ai}, \xi_{bi}).
$$
which contradicts to  (2) in the choice of $\xi_{ai}$ and $\xi_{bi}$. \qed

\begin{center}
\includegraphics[scale=0.5]{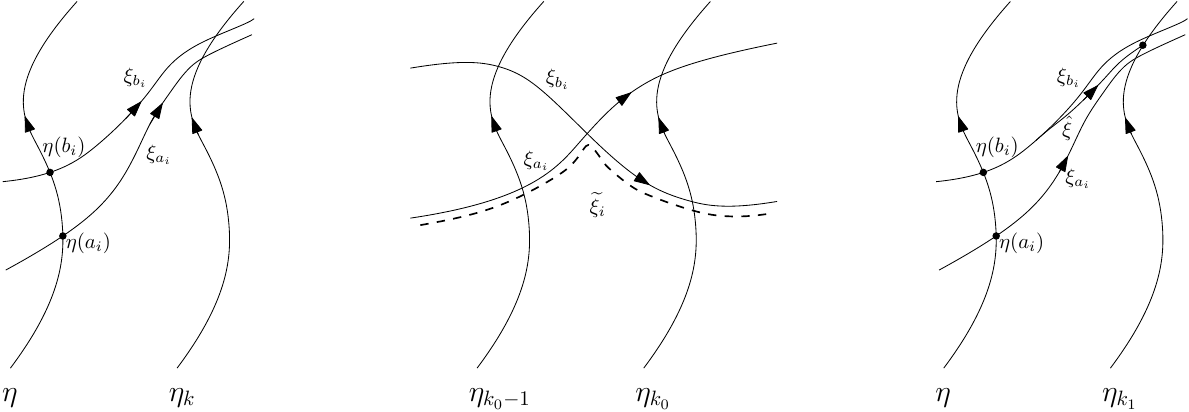}
\captionof{figure}{}
 \label{f9}
\end{center}

\begin{lem}\label{lem:intersectionsum} For all $i\in \Nset$, 
$$
\sum_{k\in \Zset}|g_{\xi_{ai}}(k)-g_{\xi_{bi}}(k)|\leq T_0.
$$
\end{lem}

Proof: Fix $i\in \Nset$. For simplicity, we omit the dependence on $i$ below.  We first establish the following fact.

\medskip

{\bf Claim:} For all $k\in \Zset$, there does NOT exists a $\xi\in \mathcal{C}_{p_0}$ such that 
$$
g_{\xi}(k)\in (g_{\xi_{ai}}(k), g_{\xi_{bi}}(k)).
$$
 We argue by contradiction.  If not, then there exists $k_1\in \Nset$ and  $\hat \xi\in \mathcal{C}_{p_0}$ such that 
$$
g_{\hat \xi}(k_1)\in (g_{\xi_{ai}}(k_1), g_{\xi_{bi}}(k_1)).
$$
Since \(\hat{\xi}\) cannot pass through the segment of \(\eta\) over the interval \((a_0, b_0)\), it must intersect either \(\xi_{ai}\) (or \(\xi_{bi}\)) somewhere between \(\eta_0\) and \(\eta_{k_1}\). By gluing \(\hat{\xi}\) with \(\xi_{ai}\) (or \(\xi_{bi}\) correspondingly), in a manner similar to the construction in (\ref{eq:glueirrational}), we can obtain a new curve that satisfies condition (1) but has a smaller distance to \(\xi_{bi}\) (or \(\xi_{ai}\)). See the right picture on Figure \ref{f9} above.  This contradicts the optimality of the choice of \(\xi_{ai}\) and \(\xi_{bi}\).

Therefore, for fixed $k\in \Zset$, we could  choose a suitable integer $l_k$ such that for some $j_k\in \Nset$
$$
I_{k}=(g_{\xi_{ai}}(k), g_{\xi_{bi}}(k))+l_{k}T_0= (a_{j_k},b_{j_k})\subset [0,T_0].
$$

In addition, if $I_{k_1} = I_{k_2}$ for $k_1 \neq k_2$, then there exist $t_2 > t_1$ such that  
\[
\xi_{ai}(t_2) - \xi_{ai}(t_1) \in \Zset^2,
\]
which by part (2) of Lemma~\ref{lem:direction} implies that $n_{p_0} \cap \Rset \Zset^2 \neq \emptyset$.  
This is impossible.  
Hence the intervals $I_{k}$ are mutually disjoint for $k \in \Zset$, and the conclusion follows.  
\qed

 \medskip

Finally,  we are ready to prove the $p_0$ is a nonlinear point.  We argue by contradiction. If not, assume that $p_0$ lies on a line segment on $\partial F_0$.  Let $p$, $p'$ be two distinct points in the interior of the line segment.  Owing to Lemma \ref{lem:equal-aubry}, 
$$
\tilde {\mathcal{A}}_{p}=\tilde {\mathcal {A}}_{p'},
$$
which leads to $(p-p')\cdot q_0=0$,  $\mathcal{C}_p=\mathcal{C}_{p'}$, $\mathcal{I}_p=\mathcal{I}_{p'}$. 

Let $v_p$ and $v_{p'}$ be corresponding  viscosity solutions to the cell problem (\ref{mech-cell}). Then
$$
Du_p=Du_{p'} \quad \text{on $\mathcal{A}_{p}$}.
$$
Here $u_p=p\cdot x+v_p$ and $u_{p'}=p'\cdot x+v_{p'}$.  Write
$$
h(t)=u_{p}(\eta_0(t))-u_{p'}(\eta_0(t))=(p-p')\cdot \eta(t)+v_p(\eta_0(t))-v_{p'}(\eta_0(t)).
$$
Then
$$
h'(t)=0  \quad \text{for $t\in \mathcal{I}_p$}.
$$
Moreover, since \((p - p') \cdot q_0 = 0\) and \(q_0\) is not parallel to $(\widetilde{m},\widetilde{n})$, it follows that
\begin{equation}\label{g-infty}
\lim_{t \to +\infty} |h(t)| = +\infty.
\end{equation}
Besides, due to Lemma \ref{lem:intersectionsum}, 
$$
\lim_{|k|\to +\infty}|g_{\xi_{ai}}(k)-g_{\xi_{bi}}(k)|=0.
$$
Note that
$$
\begin{array}{ll}
&u_p(\eta_0(g_{\xi_{ai}}(k))-u_p(\eta_0(a_i))=u_{p'}(\eta_0(g_{\xi_{ai}}(k))-u_{p'}(\eta_0(a_i))\\[3mm]
&u_p(\eta_0(g_{\xi_{bi}}(k))-u_p(\eta_0(b_i))=u_{p'}(\eta_0(g_{\xi_{bi}}(k))-u_{p'}(\eta_0(b_i)).
\end{array}
$$
Sending $k\to \infty$ leads to that for each $i\in \Nset$,
$$
h(a_i)=h(b_i).
$$
Define a new function
$$
l(t)=
\begin{cases}
h(t) \quad \text{for $t\in \mathcal{I}_p$}\\
h(a_i) \quad  \text{for $t\in (a_i, b_i)$ and $i\in \Nset$}.
\end{cases}
$$
Clearly, $l(t)$ is Lipschitz continuous and $l'(t)=0$ for a.e. $t\in \Rset$. Hence $l(t)\equiv c$ for some constant $c$, which is absurd owing to (\ref{g-infty}). \qed

\section{Proof of Theorem \ref{main2}}. 

Throughout this section, we  assume  that $p_0\in \partial F_0$ and $n_{p_0}$ contains more than one elements. Let us first establish the following Lemma.

\begin{lem} We have that 
\be\label{combination}
n_{p_0}=\left\{{t{v}_0+(1-t){v}_1\over |t{v}_0+(1-t){v}_1|}|\ t\in [0,1]\right\}.
\ee
 ${v}_0,{v}_1\in G_p$ are two distinct primitive integer vectors such that  $G_{p_0}=\{{v}_0,\ {v}_1\}$ or $G_{p_0}=\{{v}_0,\ {v}_1,\  {v}_0+{v}_1\}$ and $\mathrm{det}(v_1,v_0)=1$.  Here $(v_1,v_0)$ denotes the $2\times 2$ matrix whose first and second columns are $v_1^{\top}$ and $v_0^\top$ (transposes of $v_1$ and $v_0$) respectively. 
\end{lem}

Proof: By part (3) of Lemma \ref{lem:G-property}, it suffices to  shows that for any \(q \in n_{p_0} \cap \mathbb{R}\mathbb{Z}^2\), there exist $m\in \Nset$, $\lambda_1,\lambda_2,..,\lambda_m>0$ such that
\be\label{eq:Gpcombination}
q = \sum_{k=1}^{m}\lambda_k u_k
\ee
for $u_1,u_2,..,u_m\in G_{p_0}$.

In fact, given two arbitrary distinct unit vectors $q_1,q_2\in \Rset \Zset^2\cap n_{p_0}$.  For $i=1,2$, write
$$
q_i=\lambda w_i
$$
for some $\lambda>0$ and a primitive integer vector $w_i$. Thanks to  Corollary \ref{cor:aubrynonempty},  for $i=1,2$,  there are two periodic (modulo $\Zset^2$) ECCUS $\xi_{\infty, i}$ on $\mathcal {A}_{p_0}$,  each with first homology class $w_i$.

Since $q_1$ and $q_2$ are not parallel,  $\xi_{\infty, 1}$ and $\xi_{\infty, 2}$  must intersect. As a result,  $\xi_{\infty, 1}(\Rset)\cap \xi_{\infty, 2}(\Rset)\cap \Zset^2\not= \emptyset$. Accordingly, without loss of generality, we may assume that for $i=1,2$
$$
\xi_{\infty, i}(T_{\infty, i})=w_i, \quad \xi_{\infty, i}(0)=(0,0).
$$
for some $T_{\infty, i}>0$.  Due to (2) in Lemma \ref{lem:cusproperty}, $G_p\not=\emptyset$ and $\xi_i$ are union of unit-speed reparametrized homoclinic orbits. Thus, $w_i$ can be written as a sum of elements in $G_p$ for $i=1,2$, which implies (\ref{eq:Gpcombination}).\qed

Hereafter, we may assume that 
\be\label{eq:twocaseg}
G_{p_0}=\{{v}_0,\ {v}_1\} \quad \mathrm{or}\quad G_{p_0}=\{{v}_0,\ {v}_1,\  {v}_0+{v}_1\},
\ee
where $v_0$ and $v_1$ satisfy that $\mathrm{det}(v_1^\top,v_0^\top)$. Let \(\gamma_0\) and \(\gamma_1\) be homoclinic orbits associated with \(p_0\) such that
\[\lim_{t\to -\infty}\gamma_0(t)=\gamma_1(t)=(0,0),\ \lim_{t\to \infty}\gamma_0(t)=v_0\ \mathrm{and}\ \lim_{t\to \infty}\gamma_1(t)=v_1.
\]
 Write   ${v}_0=(m_0,n_0)$ and ${v}_1=(m_1,n_1)$. Then 
\be\label{eq:det1}
\mathrm{det}(v_1^\top,v_0^\top)=\det
\begin{pmatrix}
m_1 & m_0 \\
n_1 & n_0
\end{pmatrix}
= m_1 n_0 - m_0 n_1=1.
\ee

\medskip

{\bf Proof of Theorem \ref{main2}.} Let us establish (\ref{eq:coneequation}) for $i=0$. The proof for $i=1$ is similar.

We argue by contradiction. Assume that it is not true. Then {\bf Assumption 1} and {\bf Assumption 2} in Section \ref{sec:main} holds.

Consider the periodic extension (see (\ref{eq:extension})) and unit-speed reparametrization of \(\gamma_0\) and \(\gamma_1\), which we continue to denote by \(\gamma_0\) and \(\gamma_1\) for convenience. By suitable translation in time, we may assume that for all $t\in \Rset$ and $k\in \Zset$,
\[
\gamma_0(0) = \gamma_1(0) = (0,0), \ \gamma_0(t+kT_0) =\gamma_0(t)+kv_0, \ \text{and} \ \gamma_1(t+kT_1) =\gamma_1(t)+kv_1
\]
for some \(T_0, T_1 > 0\).

For $i=0,1$, let $\Sigma_i$ be the collection of all periodic (modulo $\Zset^2$) ECCUS $\xi_i$ with first homology class $(m_0,n_0)$ and $\xi_i(0)=-i(m_1,n_1)$.  Clearly, $\gamma_0\in \Sigma_0$ and $\gamma_0-(m_1,n_1)\in \Sigma_1$ and for $i=0,1$,
$$
\xi_i(\Rset)\cap \Zset^2=\{\xi_i(0)+k(m_0,n_0)|\ k\in \Zset\}.
$$
Each $\xi_i$ consists of reparameterized homoclinic orbits. 

We intend to apply Lemma ~\ref{lem:approxgap}, which requires to identify the gap and verify {\bf Assumption 3} for suitable choices from $\Sigma_0$ and $\Sigma_1$.

Note that for $\xi_0\in \Sigma_0$ and $\xi_1\in \Sigma_1$, $\xi_0(\Rset)\cap \xi_1(\Rset)=\emptyset$ and the open strip $U$ between $\xi_0$ and $\xi_1$ does not contain any integer point. Otherwise, an integer translation of $\xi_0$ will intersect the $\gamma_1$ away from $\Zset^2$, which is impossible. 

Define the quantity
$$
I(\xi_0,\xi_1)=\int_{U}{1\over 1+|x|^4}\,dx. 
$$
By the stability of ECCUS and its smoothness away from $\Zset^2$, we may choose $\tilde \xi_0\in \Sigma_0$ and $\tilde \xi_1\in \Sigma_1$ such that 
$$
I(\widetilde \xi_0,\widetilde \xi_1)=\min_{\xi_0\in \Sigma_0, \ \xi_1\in \Sigma_1}I(\xi_0,\xi_1).
$$
Then  open strip $\widetilde U$ between  $\widetilde \xi_0$ and $\widetilde \xi_1$ is a {\it ``gap"} with the direction $(m_0,n_0)$ according to Definition \ref{def:gap}. 

Let $\Omega_+$ be the domain from \eqref{eq:twodomainsep} with $(m,n)$ replaced by $(m_0,n_0)$ and $\xi_0$ replaced by $\widetilde\xi_0$. 

{\bf Claim:}
\be\label{eq:assumtion3}
\tilde\xi_1(\Rset)\subset \Omega_+.
\ee

In fact, by \eqref{eq:det1}, for $k$ sufficiently large we have
\[
\tilde\xi_0(\Rset)-k(m_1,n_1)\subset \Omega_+.
\]
Meanwhile, since for each $k\in\Zset$, the open strip bounded by the two translates  $\widetilde\xi_0+(k-1)(m_1,n_1)$ and $\widetilde\xi_0+k(m_1,n_1)$ contains no integer points, $\widetilde{\xi_0}$ does not lie in any of those open strips for $k\geq 1$.  As a result, $\widetilde\xi_0+k(m_1,n_1)\subset \Omega_+$ for all $k\geq 1$. So (\ref{eq:assumtion3}) holds.

We are now in a position to apply  Lemma~\ref{lem:approxgap}. Let $\gamma_+$ denote the resulting curve.  Since $\gamma_+$ can not intersect $\gamma_1$ within $\widetilde{U}$, we must have (by suitable translation in space  if needed) 
$$
\gamma_+(t_+)=lv_0-(m_1,n_1)\in \tilde \xi_1(\Rset), \quad   \gamma_+(t_-)=(0,0) \quad \mathrm{and} \quad \gamma_+((t_-,t_+))\subset \widetilde{U}
$$
for some $t_+>0>t_-$ and $l\in \Zset$. Hence $lv_0-(m_1,n_1)\in G_{p_0}$, which implies that $lv_0-(m_1,n_1)=l_0v_0+l_1v_1$ for $l_0,l_1\in \{0,1\}$. This is absurd. \qed

\section{Local behavior near the critical point}
Assume that $W\in C^{\infty}(\Rset^2)$, $W(O)=0$ and  $W$ has a local maximum at the origin $O=(0,0)$. Suppose that the Hessian $D^2W(O)$ has two distinct negative eigenvalues $0>-a>-b$. If $\eta:[T_1,T_2]\to \Rset^2$ is a Lipschitz continuous curve, we denote
\be\label{eq:LW}
L_W([T_1,T_2],  \eta)=\int_{T_1}^{T_2}{1\over 2}|\dot \eta|^2-W(\eta)\,dt.
\ee

\begin{lem}\label{near-origin} Assume that 
$$
W(x_1,x_2)=-{a^2\over 2}x_{1}^{2}-{b^2\over 2}x_{2}^{2}.
$$
Let $\lambda={a\over b}\sqrt{b-a\over b+a}$. Then for $\beta_1, \beta_2\in (0,\lambda]$, $s_1<0<s_2$ and
$$
P_1=(s_1, -\beta_1 s_1)   \quad \mathrm{and} \quad P_2=(s_2, \beta_2 s_2).
$$
$$
\begin{array}{ll}
L_W([0,+\infty), \eta_1)+L_W((-\infty, 0], \eta_2)&=\frac{1}{2} \left[ a\left( s_1^2 + s_2^2 \right) + b\left( \beta_1^2 s_1^2 + \beta_2^2 s_2^2 \right) \right]\\[5mm]
&<\int_{0}^{T}{1\over 2}|\dot \xi|^2-W(\xi(t))\,dt
\end{array}
$$
Here $\eta_1=(s_1e^{-at}, -\beta_1s_1e^{-bt})$ and  $\eta_2=(s_2e^{at},  \beta_2s_2e^{bt})$.  $\xi:[0,T]\to \Rset^2$ is any curve satisfying that 
$$
\begin{cases}
\ddot \xi=-DW(\xi)\\
{1\over 2}|\dot \xi|^2+W(\xi(t))\equiv 0\\
\xi(0)=P_1,  \quad \xi(T)=P_2.
\end{cases}
$$

\end{lem}
\begin{center}
\includegraphics[scale=0.9]{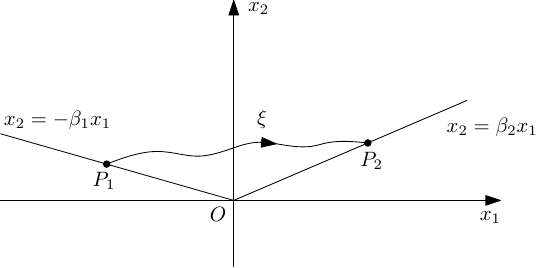}
\captionof{figure}{}
 \label{f10}
\end{center}

\nit Proof: The  $``="$ comes from direction computation. We only need to establish $``<"$.  By translation,  we can set 
$$
\xi(t)=(\nu(e^{at}-e^{-at}),  \  ce^{bt}+de^{-bt})
$$
satisfying that 
\be\label{eq:twoendpoint}
\xi(-T_1)=P_1   \quad \mathrm{and} \quad \xi(T_2)=P_2
\ee
for some $T_1,T_2>0$. Then $s_2>0$ lead to $\nu>0$ and
\be\label{eq:abcd}
{1\over 2}|\dot \xi|^2+W(\xi(t))\equiv 0  \quad \Rightarrow \quad \nu^2a^2=cdb^2
\ee
and by (\ref{eq:twoendpoint}),
\be\label{eq:endmatch}
\begin{cases}
{ce^{bT_2}+de^{-bT_2}\over \nu(e^{aT_2}-e^{-aT_2})}=\beta_2\leq \lambda\\[5mm]
{ce^{-bT_1}+de^{bT_1}\over \nu(e^{aT_1}-e^{-aT_1})}=\beta_1\leq \lambda
\end{cases}
\ee

Due to (\ref{eq:abcd}), $cd>0$.  Together with $s_2\beta_2>0$, we have that $c,d>0$.   Accordingly, (\ref{eq:endmatch}) implies that
\be\label{eq:controlcd}
0\leq c\leq \nu \beta_2 e^{(a-b)T_2}\leq \nu \lambda, \quad 0\leq d\leq \nu \beta_1 e^{(a-b)T_1}\leq \nu \lambda.
\ee
Moreover,  for $i=1,2$
$$
\lambda\geq  {2\sqrt{cd}\over \nu(e^{aT_i}-e^{-aT_i})}={2a\over  b(e^{aT_i}-e^{-aT_i})}\geq {2a\over  be^{aT_i}}
$$
Thus, for $i=1,2$,
\be\label{eq:controlnegative}
e^{-aT_i}\leq {b\lambda\over 2a}.
\ee
A straightforward calculation shows that 
$$
\int_{-T_1}^{T_2}{1\over 2}|\dot \xi (t)|^2-W(\xi(t))\,dt={\bf I}+{\bf II}
$$
Here
$$
{\bf I}={1\over 2}a\nu^2(e^{2aT_2}-e^{-2aT_1})+{1\over 2}a\nu^2(e^{2aT_1}-e^{-2aT_2})
$$
and
$$
{\bf II}={1\over 2}bc^2(e^{2bT_2}-e^{-2bT_1})+{1\over 2}bd^2(e^{2bT_1}-e^{-2bT_2}).
$$
Recall that for $i=1,2$,  
\[
s_i^{2}
  = \nu^{2}\!\bigl(e^{aT_i}-e^{-aT_i}\bigr)^{2},
\qquad
\beta_1^{2}s_1^2
  = \bigl(c\,e^{-bT_1}+d\,e^{bT_1}\bigr)^{2}, \qquad \beta_2^{2}s_2^2
  = \bigl(c\,e^{bT_2}+d\,e^{-bT_2}\bigr)^{2}.
\]

It is easy to see that
$$
{\bf I}+{\bf II}=\frac{1}{2} \left[ a\left( s_1^2 + s_2^2 \right) + b\left( \beta_1^2 s_1^2 + \beta_2^2 s_2^2 \right) \right]+{\bf III}.
$$
Here
$$
\begin{array}{ll}
{\bf III}&=2(a\nu^2-bcd)-a\nu^2(e^{-2aT_1}+e^{-2aT_2})-bc^2e^{-2bT_1}-bd^2e^{-2bT_2}\\[5mm]
&> 2(a\nu^2-bcd)-{1\over 2}{b^2\over a}\nu^2\lambda^2-bc^2-bd^2\quad \text{by (\ref{eq:controlnegative})}\\[5mm]
&\geq 2(a\nu^2-{a^2\nu^2\over b})-{1\over 2}{b^2\over a}\nu^2\lambda^2-2b\lambda^2\nu^2\quad \text{by (\ref{eq:abcd}) and (\ref{eq:controlcd})}\\[5mm]
&\geq  (a\nu^2-{a^2\nu^2\over b})>0
\end{array}
$$
if we choose 
$$
\lambda^2={a^2(b-a)\over b^2(b+4a)}.
$$
\qed

Assume that $u\in C^{0,1}(\Rset^2)$ is a viscosity subsolution to 
$$
{1\over 2}|Du|^2+W(x)=0  \quad \text{in $\Rset^2$}.
$$
subject to $u(0,0)=0$.  Clearly,
\be\label{eq:controlofu}
|u(x)|\leq C|x|^2  \quad \mathrm{and} \quad |Du(x)|\leq C|x|.
\ee
\begin{theo} \label{theo:no-flat-fill}Assume that
\[
W(x_1,x_2) = -\frac{a^2}{2}x_1^2 - \frac{b^2}{2}x_2^2 + O(|x|^3) \quad \text{as } x \to (0,0).
\]
Then there does not exist a sequence of positive numbers \( \{a_{1n}\}_{n \geq 1} \) and \( \{a_{2n}\}_{n \geq 1} \), together with a sequence of characteristics \( \xi_n : [0, T_n] \to \mathbb{R}^2 \) of \( u \), such that

\medskip

(i) $\lim_{n\to \infty}a_{1n}=\lim_{n\to \infty}a_{2n}=0$;\\

(ii) $\xi_n(0)=\left(-a_{1n}, {\lambda\over 2}a_{1n}\right)$ and $\xi_n(T_{n})=\left(a_{2n}, {\lambda\over 2}a_{2n}\right)$;

(iii) 
$$
\int_{0}^{T_{n}}{1\over 2}|\dot \xi_n|^2-W(\xi_n)\,dt=u\left(a_{2n}, {\lambda\over 2}a_{2n}\right)-u\left(-a_{1n}, {\lambda\over 2}a_{1n}\right).
$$
Here $\lambda$ is the same as that in Lemma \ref{near-origin}.
\end{theo}

\nit Proof:  We argue by contradiction. Suppose that such sequences \( \{a_{1n}\}_{n\geq 1} \), \( \{a_{2n}\}_{n\geq 1} \), and  \( \{\xi_n\}_{n\geq 1} \) do exist.

By the two dimensional  topology, symmetry with respect to the origin, and (if necessary) by selecting a suitable portion of \( \xi_n \), we may assume that

$$
\xi_n([0,T_n])\subset D^{+}=\left\{(x_1,x_2)|\  x_{2}\geq{ \lambda\over 2}|x_1|\right\}.
$$

Without loss of generality, we assume that 
$$
r_n=|a_{1n}|\geq a_{2n}.
$$
Denote
$$
W_n(x)={W(r_nx)\over r_{n}^{2}} \quad \text{and} \quad \eta_n(t)={\xi_n(t)\over r_n}
$$
Then
$$
\lim_{n\to +\infty}W_n(x)={\widetilde W}(x)=-{a^2\over 2}x_{1}^{2}-{b^2\over 2}x_{2}^{2}  \quad \text{locally uniformly in $\Rset^2$},
$$
$$
\ddot \eta_n(t)=-DW_n(\eta_n),\quad \eta_n(0)=\left(-1,{\lambda\over 2}\right),\quad  {1\over 2}|\dot \eta_n|^2+W_n(\eta_n)\equiv 0, 
$$
and by (iii) and  (\ref{eq:controlofu}), 
\be\label{finite-energy}
\int_{0}^{T_{n}}{1\over 2}|\dot \eta_n|^2-W_n(\eta_n)\,dt={u(\xi_n(T_n))-u(\xi_n(0))\over r_n^2}\leq C. 
\ee
Up to a subsequence if necessary, we may assume that 
$$
\lim_{n\to \infty}T_n=T,  \quad \lim_{n\to \infty}{a_{2n}\over r_n}=\tilde a, 
$$
$$
\lim_{n\to \infty}{\eta_n}=\eta
$$
{\bf Claim: $T<\infty$ and $\tilde a>0$}.

We argue by contradiction. Suppose \( T = \infty \). Then $\eta([0, \infty)) \subset D^+$, 
\be\label{eq:etaprop}
\ddot{\eta}(t) = -D{\widetilde W}(\eta(t)), \quad \eta(0) = \left(-1, \frac{\lambda}{2}\right), \quad  \frac{1}{2}|\dot{\eta}(t)|^2 + {\widetilde W}(\eta(t)) \equiv 0.
\ee
Moreover, 
\[
\int_{0}^{\infty} \sqrt{-2{\widetilde W}(\eta(t))} \, |\dot{\eta}(t)| \, dt 
\ \leq\ 
\int_{0}^{\infty} \left( \frac{1}{2} |\dot{\eta}(t)|^2 - {\widetilde W}(\eta(t)) \right) dt 
\ \leq\ C,
\]
which implies that \( \lim_{t \to \infty} \eta(t) = (0, 0) \). Consequently, 
\[
\eta(t) = \left(-e^{-at}, \, \frac{\lambda}{2} e^{-bt}\right).
\]
This contradicts  \( \eta([0, \infty)) \subset D^+ \). Hence, \( T < \infty \).

If \( \tilde{a} = 0 \), then \( \eta(T) = (0, 0) \), which would imply \( \dot{\eta}(T) = (0, 0) \) by the last equality in (\ref{eq:etaprop}), and thus \( \eta(t) \equiv (0, 0) \) for all \( t \in [0, T] \), a contradiction. Therefore, we must have \( \tilde{a} > 0 \).

Upon a subsequence if necessary,  we may assume that 
$$
\lim_{n\to +\infty}{u(r_nx)\over r_{n}^{2}}=\tilde u(x)  \quad \text{locally uniformly}.
$$
Clearly, $\tilde u\in C^{0,1}(\Rset^2)$ is a viscosity subsolution of 
$$
{1\over 2}|D\tilde u|^2-\widetilde W=0  \quad \text{in $\Rset^2$}
$$
and
$$
\int_{0}^{T}{1\over 2}|\dot {\eta}|^2-{\widetilde W}(\eta)\,dt=\tilde u\left(\tilde a,{\lambda\over 2}\tilde a\right)-\tilde u\left(-1, {\lambda\over 2}\right)
$$
subject to $\eta(0)=\left(-1, {\lambda\over 2}\right)$ and $\eta(T)=\left(\tilde a,{\lambda\over 2}\tilde a\right)$.  Due to Lemma \ref{calib}, this implies that $\eta$ is an absolutely minimizing curve connecting $\left(-1, {\lambda\over 2}\right)$ and $\left(\tilde a,{\lambda\over 2}\tilde a\right)$. This contradicts to Lemma \ref{near-origin}. 

\qed

\begin{lem}\label{lem:nodirectionoverlap} For $p\in F_0$,  there does not exists two orbits $\gamma_+$ and $\gamma_-$ on $\mathcal{A}_p$ such that
$$
\lim_{t\to \infty}\gamma_+(t)=\lim_{t\to -\infty}\gamma_-(t)=O=(0,0)
$$
and the two limiting  directions coincide
\be\label{eq:directionequal}
\lim_{t\to \infty}{\gamma_+(t)\over |\gamma_+(t)|}=\lim_{t\to -\infty}{\gamma_-(t)\over |\gamma_-(t)|}. 
\ee
\end{lem}
\begin{center}
\includegraphics[scale=0.3]{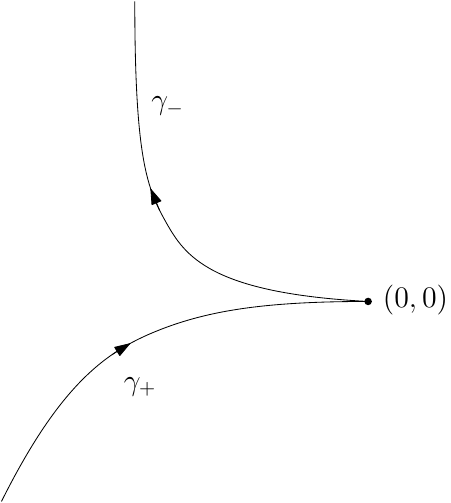}
\captionof{figure}{}
 \label{f11}
 \end{center}

Proof: Note that the existence of these two limits are ensured by Lemma \ref{lem:vunique}.   We argue by contradiction by assuming that they are the same. 

Note that there exists $r_0>0$ such that
$$
V(x)\leq \widetilde V(x)=-{a^2\over 4}|x|^2  \quad \text{for $x\in B_{r_0}(O)$}.
$$
Without loss of generality, we may assume that 
$$
\gamma_{+}([0,\infty))\cup \gamma_{-}([0,\infty))\subset B_{r_0}(O).
$$
Accordingly for $T,T'>0$,
\be\label{eq:twoVcomparison}
\begin{array}{ll}
L_V([T,+\infty), \gamma_+)\geq L_{\widetilde V}([T,+\infty), \gamma_+)\\[3mm]
L_V((-\infty, -T'], \gamma_-)\geq L_{\widetilde V}((-\infty, -T'], \gamma_-).
\end{array}
\ee
See (\ref{eq:LW}) for the definition of $L_V$ and $L_{\widetilde{V}}$. 
Note that for $a < b$ and $\xi \in AC([a,b], B_{r_0}(O))$,
$$
\int_{a}^{b}{1\over 2}|\dot \xi|^2+{a^2\over 4}|\xi|^2\,dt\geq {\sqrt{2}a\over 2}\int_{a}^{b}|\dot \xi\cdot \xi|\,dt\geq {\sqrt{2}a\over 4}\left||\xi(b)|^2-|\xi(a)|^2\right|. 
$$
Accordingly, 
\be\label{eq:Vtildecontrol}
\begin{array}{ll}
L_{\widetilde V}([T,+\infty),\gamma_+)+L_{\widetilde V}((-\infty, -T'], \gamma_-)\geq {a\sqrt{2}\over 4}(\bigl|\gamma_+(T)\bigr|^2+\bigl|\gamma_-(-T')\bigr|^2).
\end{array}
\ee

Now let $v$ be a viscosity solution of \eqref{mech-cell} and set $u(x)=p\cdot x+v(x)$. For $m\in\Nset$, choose $T_m,T'_m>0$ with $T_m,T'_m\to\infty$ as $m\to\infty$ such that
\be\label{eq:tworange}
\begin{array}{ll}
&\bigl|\gamma_+(T_m)\bigr|=\bigl|\gamma_-(-T'_m)\bigr|=\frac{1}{m}\\[3mm]
&\gamma_+([T_m,\infty))\cup\gamma_-((-\infty,-T'_m])\subset B_{1\over m}(O).
\end{array}
\ee
Then
$$
\begin{array}{ll}
u(\gamma_-(-T_m'))-u(\gamma_+(T_m))&=L([T_m,+\infty), \gamma_+)+L((-\infty, -T_m'], \gamma_-)\\[3mm]
&>{a\sqrt{2}\over 2m^2} \quad \text{by (\ref{eq:twoVcomparison}) and (\ref{eq:Vtildecontrol})}. 
\end{array}
$$
Meanwhile, owing to (\ref{eq:controlofu}) and (\ref{eq:tworange}),
$$
\begin{array}{ll}
u(\gamma_-(-T_m'))-u(\gamma_+(T_m))&\leq {C\over m}|\gamma_-(-T_m')-\gamma_+(T_m)|\\[3mm]
&={o(1)\over m^2} \quad \text{by (\ref{eq:directionequal})},
\end{array}
$$
where $\lim_{m\to \infty}o(1)=0$. This is a contraction.
\qed

. 

As a corollary, we obtain a necessary condition for when \(\mathbb{R}^2\) can be filled with periodic ECCUS. A more general conclusion under additional generic assumptions was established in Proposition~3.1 of \cite{Cheng}; however, those assumptions are not required in our setting.

\begin{cor}\label{cor:uniqueaubry} For $p_0\in  \partial F_0$,  assume that $n_{p_0}=\{q_0\}$ for some $q_0\in \Rset\Zset^2$.  If there is no gap between every two periodic (modulo $\Zset^2$) ECCUS associated with $p_0$, then 
$$
\gamma_{b0}(t),  \quad \gamma_{b0}(-t), \quad  \gamma_{b1}(t),  \quad \gamma_{b1}(-t).
$$
are all homoclinic orbits associated with either $p_0$ or $-p_0$.
\end{cor}

Proof: Write \( q_0 = \lambda_0(m_0,n_0) \) for $\lambda_0>0$ and $(m_0,n_0)$ is a primitive integer vector.  As in (\ref{eq:P_0}),  let $\mathcal{P}_0$ be the collection of  all periodic (modulo \( \mathbb{Z}^2 \)) ECCUS associated with $p_0$ with first homology class \( (m_0,n_0)\). By Corollary~\ref{cor:aubrynonempty}, $\mathcal{P}_0\not=\emptyset$. Since $n_{p_0}$ has a single vector, due to Lemma \ref{lem:direction}, a curve $\xi\in \mathcal{P}_0$ is either a unit reparameterization of the periodic expansion (\ref{eq:extension}) of a homoclinic orbit or a unit reparametrization of a periodic orbit on $\mathcal{M}_{p_0}$.

Since there is no gap with the direction $(m_0,n_0)$,  there must exist $\{\xi_n\}_{n\geq 1}\subset \Sigma$ such that $\lim_{n\to \infty}\xi_n(0)=(0,0)$, which leads to the existence of a periodic (modulo \( \mathbb{Z}^2 \)) ECCUS that contain $k(m_0,n_0)$ for all $k\in \Zset$, which must be a unit reparameterization of the periodic expansion (\ref{eq:extension}) of a homoclinic orbit with first homology class $(m_0,n_0)$. 

For \(i = 0,1\), let \(\Sigma_i\) denote the set of periodic (modulo \(\mathbb{Z}^2\)) ECCUS $\xi_i$ with first homology class \((m_0,n_0)\) and
\[
\xi_i(0) = (0,0) + i(m_1,n_1).
\]
Here $(m_1,n_1)$ is the integer vector satisfying $(m_1,n_1)\cdot (-n_0,m_0)=1$.

According to the previous discussions, $\Sigma_i\not=\emptyset$ for $i=0,1$.  For  $\xi_0\in \Sigma_0$ and $\xi_1\in \Sigma_1$, $\xi_0(\Rset)\cap \xi_1(\Rset)=\emptyset$.  Let $U$ be the open strip bounded by $\xi_0$ and $\xi_1$. Clearly, $U\cap \Zset^2=\emptyset$. 

As in  the proof of Theorem \ref{main2}, we introduce the quantity
$$
I(\xi_0,\xi_1)=\int_{U}{1\over 1+|x|^4}\,dx. 
$$
By the stability of ECCUS and their smoothness away from $\Zset^2$, we may choose $\eta_0\in \Sigma_0$ and $\eta_1\in \Sigma_1$ such that 
$$
I(\eta_0,\eta_1)=\min_{\xi_0\in \Sigma_0, \ \xi_1\in \Sigma_1}I(\xi_0,\xi_1).
$$
Let $\widetilde{U}$ be the region between $\eta_0$ and $\eta_1$. Clearly, $\widetilde{U}$ does not contain any homoclinic orbit.  As in (\ref{eq:twodomainsep}), $\eta_0$ separats $\Rset^2$ into $\Omega_+$ and $\Omega_-$. Without loss of generality, let us assume that $\eta_1(\Rset)\subset \Omega_+$. The proof for the other case is similar. 

Therefore, if $\widetilde{U}$ is not a gap,  we must have unbounded periodic (modulo $\Zset^2$) orbits on $\mathcal{M}_{p_0}$ lying in $U$ approaching $\eta_0$ and $\eta_1$. 

Due to Lemma~\ref{lem:vunique}, the following four limits exist:
\[
v_0^{\pm} = \lim_{s \to 0^{\pm}} \frac{\eta_0(s)}{|\eta_0(s)|}, 
\quad 
v_1^{\pm} = \lim_{s \to 0^{\pm}} \frac{\eta_1(s)-(0,1)}{|\eta_1(s)-(0,1)|},
\]
and they all belong to the set \( \{ \pm v_a, \pm v_b \} \). Note that $\to 0^-$ and $\to 0^+$ are corresponding to stable and unstable directions respectively.

From now on, it suffices to examine the local behavior near the origin \( (0,0) \). By applying a suitable rotation, we may assume that
\[
V(x) = -\frac{a^2}{2} x_1^2 - \frac{b^2}{2} x_2^2 + O(|x|^3),
\]
with \( v_a = (1,0) \) and \( v_b = (0,1) \).

Note that $\xi(t)$ on  $\mathcal{A}_p$ if and only if  $\xi(-t)$ on  $\mathcal{A}_{-p}$. The  uniqueness part in Lemma~\ref{lem:vunique} implies that if $v_b \in \{v_0^{+},v_0^{-},v_1^{+},v_1^{-}\}$, then $\gamma_{0b}$ is a homoclinic orbit on $\mathcal{A}_p$ or $\mathcal{A}_{-p}$. Similar statement holds for $-v_b$ and $\gamma_{1b}$.

To complete the proof, we argue by contradiction. Suppose the statement of the lemma is false.  Combining with  Theorem~\ref{theo:no-flat-fill},  Lemma \ref{lem:nodirectionoverlap} and the discussion above, we have that for $i=0,1$
$$
v_i^{+}\not=v_i^-,\quad \{v_i^{+},v_i^{-}\}\cap \{v_b,\ -v_b\}\not=\emptyset,\quad   \{v_b,\ -v_b\}\not\subseteq \{v_0^{+},v_0^{-},v_1^{+},v_1^{-}\}.
$$
Accordingly, without loss of generality, we may assume that 
$$
v_i^{+}\not=v_i^-,\quad v_b\in \{v_i^{+},v_i^{-}\},\quad  -v_b\notin \{v_0^{+},v_0^{-},v_1^{+},v_1^{-}\}.
$$

It suffices to consider the following two cases. Other cases are similar. 

\medskip

\noindent
\textbf{Case 1:} \( v_0^-=v_1^-= v_b \). \\
Then due to the uniqueness park in Lemma~\ref{lem:vunique}, we have that $\gamma_{b0}(t)$ is a homoclinic orbit associated with $p_0$, and $\eta_0$ is the unit speed reparameterization of its extension and \( \eta_1 = \eta_0 + (1,0) \). Then $v_0^+=v_1^+=v_a$ (or $-v_a$). Accordingly, periodic orbits in $\mathcal{M}_{p_0}$ approaches $\eta_0$ in both $\Omega_+$ (within $\widetilde{U}$) and $\Omega_-$ (translation of those approaching $\eta_1$ in $\widetilde{U}$). Then Theorem~\ref{theo:no-flat-fill} will be violated from one of directions. See Figure \ref{f12} below.

\medskip

\noindent
\textbf{Case 2:} \( v_0^-=v_b \) and \( v_1^+= v_b \). \\
Then  due to the uniqueness part in Lemma~\ref{lem:vunique}, both $\gamma_{b0}(t)$ and $\gamma_{b0}(-t)$ are homoclinic orbits on $\mathcal{A}_{p_0}$, which is absurd since \( \gamma(t) \) is a homoclinic orbit associated with \( p_0 \) if and only if \( \gamma(-t) \) is a homoclinic orbit associated with \( -p_0 \).

\medskip

In both cases we reach a contradiction. 

\begin{center}
\includegraphics[scale=0.6]{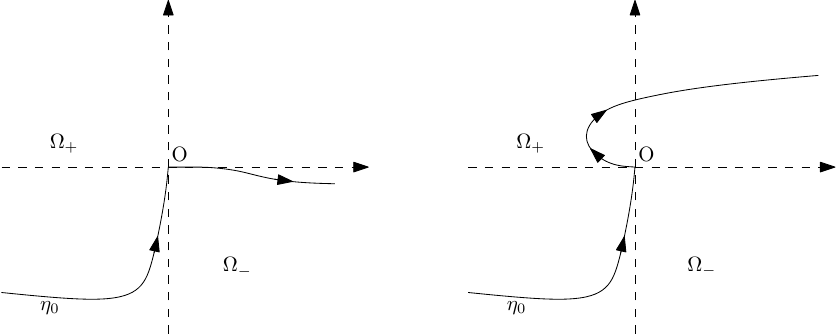}
\captionof{figure}{}
 \label{f12}
\end{center}

\qed

\section{Examples}

In this section, we will discuss several interesting examples.

\begin{example}\label{example:nonremovable}
    
 We first construct a $V$ such that there exists $p\in \partial F_0$ which is not a linear point but $n_p=\{(0,1)\}$, which demonstrates that our result is sharp. Consider 
$$
\tilde u(x_1,x_2)=x_2\sqrt{x_{1}^{2}+x_{2}^{2}}.
$$
It is easy to check that 
$$
{1\over 2}|D\tilde u|^2-{x_{1}^{2}\over 2}-2x_{2}^{2}=0 \quad \text{on $\Rset^2$}.
$$
Its characteristics 
$$
\Psi_{\tau}(t)=\left({\tau\over 2}(e^{t}+e^{-t}), \ {|\tau|\over 4}(e^{2t}-e^{-2t})\right).
$$
with 
$$ 
\begin{cases}
{d\Psi_{\tau}(t)\over dt}=D\tilde u(\Psi_{\tau}(t))\\[3mm]
\Psi_\tau(0)=(\tau,0)\\
\end{cases}
$$
 foliate  $\Rset^2\backslash \{x_1=0\}$ when $\tau$ runs through non-zero numbers. Then it is not hard to see that we can construct a  function $u=u(x)\in C^{\infty}(\Rset^2\backslash\{0\})$ such that

\begin{enumerate}[label=\arabic*.]
\item  $v(x)=u-x_2$ is periodic;

\item
$$
u=\tilde u \quad \mathrm{and} \quad \{x\in \Rset^2|\ Du(x)=(0,0)\}=\Zset^2 \quad  \text{in  $B_{1\over 3}(O)$}
$$
\item
$$
u(x_1,0)=0  \quad \text{for $x_1\in  \Rset$}
$$
\item  For every $x_1\not=0$, the gradient flow
$$
\begin{cases}
\dot \xi_{x_1} (t)=Du(\xi_{x_1} (t))\\
\xi_{x_1} (0)=(x_1,0)
\end{cases}
$$
is a periodic (modulo $\Zset^2$)  orbit with first homology class $(0,1)$, i.e., there exists $T_{x_1}>0$ such that 
$$
\xi_{x_1}(T_{x_1})-\xi_{x_1}(0)=(0,1).
$$
\end{enumerate}
Finally, let 
$$
V(x)=-{1\over 2}|Du|^2\in C^{\infty}(\Rset^2).
$$
Then $V(x)=-{x_{1}^{2}\over 2}-2x_{2}^{2}$ in  $B_{1\over 3}(O)$ and  $V<0$ in $\Rset^2\backslash \Zset^2$.

Denote by $\overline H$ be the effective Hamiltonian associated with ${1\over 2}|p|^2+V$.
\end{example}
\medskip

{\bf Claim:} $p_0=(0,1)$ is a nonlinear point and $n_{p_0}=\{(0,1)\}$. 

\medskip
Proof: Owing to the above (4), all gradient flows from (4) are periodic (modulo $\Zset^2$)  orbits on $\mathcal{M}_{p_0}$. So $\mathcal{M}_{p_0}=\Rset^2$. Also,  there is only one homoclinic orbit (when projected to $\Bbb T^2$) with first homology class $(0,1)$. So, by Theorem \ref{main2}, $\partial F_0$ is differentiable at $p_0$ and $n_{p_0}=\{(0,1)\}$.

Next, we prove that \( p_0 \) is a nonlinear point. We argue by contradiction. Suppose there exists a distinct point \( p' \in \partial F_0 \) such that the line segment connecting \( p_0 \) and \( p' \) lies entirely on \( \partial F_0 \). Then the vector \( p_0 - p' \) must be parallel to \( (1,0) \). Let $v'$ be  a viscosity solution to (\ref{mech-cell}) associated with $p'$. Then for $u'(x)=p'\cdot x+v'(x)$
$$
\begin{array}{ll}
\int_{0}^{T_x}{1\over 2}|\dot \xi_x|^2-V(\xi_x)\,dt&=u(\xi_x(T_x))-u(x_1,0)\\[3mm]
&=p\cdot (0,1)=p'\cdot (0,1)=u'(\xi_x(T_x))-u'(x_1,0).
\end{array}
$$
Hence by {\bf (P1)},  $\xi_x$ is also the gradient flow of $u'$. Thus
$$
p'+Dv'=p_0+Dv \quad \text{in $\Rset^2$}.
$$
This implies that $(p_0-p')\cdot x+v-v'$ is a constant, which is absurd. \qed

\begin{example}\label{example:removable}
     Here we present an example where $\partial F_0$ is $C^1$ and Case 1 in Theorem \ref{main} happens. Moreover, $\overline H$ here is nowhere differentiable along $\partial F_0$. Our example is stable under small perturbation of $V$. 
     
     For $\theta\in [-1,0]$, let $V_{\theta}$ be a smooth $\Zset^2$-periodic function satisfying that 
\begin{enumerate}[label=\arabic*.]
\item  $\max_{\Rset^2 }V_{\theta}=V_{\theta}(0)=\theta<0$ and $V_{\theta}$ is non-decreasing with respect to $\theta$ 

\item 
\be\label{eq:partoverlap}
V_{\theta}=V_{-1} \quad \text{$x\in  \Rset^2\backslash \left(B_{1\over 4}(0)+\Zset^2\right)$};
\ee

\item
$$
V_{\theta}=
\begin{cases}
-1 \quad \text{if $|x|\leq {1\over 2}$}\\
-R \quad \text{if $|x|\in [{1\over 4}, {1\over 3}]$}.
\end{cases}
$$
\end{enumerate}
Here $R$ is a fixed constant satisfying that  for $\eta(t)={1\over 2}(\cos t, \sin t)$, 
\be\label{eq:Rbig}
{\sqrt {R}\over 12}\geq \int_{0}^{T}{1\over 2}|\dot \eta|^2-V_{-1}(\eta)\,dt
\ee
We denote by $\overline{H}_{\theta}$, $\widetilde{\mathcal{M}}_{p,\theta}$, and $\mathcal{M}_{p,\theta}$ the effective Hamiltonian, the Mather set, and the projected Mather set associated with the Hamiltonian ${\tfrac{1}{2}}|p|^2+V_{\theta}$, respectively.  
Define
\[
F_{0,\theta}=\{\,p\in \mathbb{R}^2 \mid \overline{H}_{\theta}(p)=0\,\}, 
\qquad 
F_0=\{\,p\in \mathbb{R}^2 \mid \overline{H}_0(p)=0\,\}.
\]
\end{example}

{\bf Claim:} $\partial F_0$ is $C^1$ and Case 1 in Theorem \ref{main} happens for $V=V_0$. Moreover, for every $p\in \partial F_0$, $\mathcal {M}_p\backslash\{0\}\not= \emptyset$. Also, $\overline H_0$ is nowhere differentiable along $\partial F_0$.

\medskip

Proof: It suffices to show that  for all $\theta\in [-1,0)$ and for all $p\in \partial F_0$, 
$$
F_{0,\theta}=\partial F_0,  \quad \partial \overline {H}_\theta(p)\subset \partial \overline {H}_0(p) \quad  \mathrm{and} \quad \widetilde{\mathcal{M}}_{p,\theta}\subset \widetilde{\mathcal{M}}_{p,0}.
$$

 Suppose that $\xi:[0,T]\to \Rset^2$ is a Lipschitz continuous curve such that $|\xi(0)|={1\over 3}$, $|\xi(T)|={1\over 4}$ and $|\xi(t)|\in [{1\over 4}, {1\over 3}]$. Then
$$
\int_{0}^{T}{1\over 2}|\dot \xi|^2-V_\mu(\xi)\,dt\geq {1\over 288 T}+RT\geq {\sqrt {R}\over 12}.
$$
Hence,  for any $\theta\in (-1,0)$ and $p\in \Rset^2$ with $\overline H_{\theta}(p)=0$, an unbounded absolutely minimizing curve associated with $L_{\theta}(q,x)={1\over 2}|q|^2-V_\theta$ can not enter $B_{1\over 4}(0)$. Otherwise, it has to reach $\partial B_{1\over 2}$ and a route along $\partial B_{1\over 2}$ has less action by (\ref{eq:Rbig}). See Figure \ref{f13} below. Therefore, for every $p\in \{p\in \Rset^2|\ \overline H_{\theta}(p)=0\}$ and all $\theta\in (0,1]$,
\be\label{eq:Matherhole}
\mathcal{M}_{p,\theta}\cap \left(B_{1\over 4}(0)+\Zset^2\right)=\emptyset.
\ee

 Since  $0>\theta=\min_{\Rset^2}\overline H_{\theta}$,  by (\cite{B1,C}), for all  $\theta\in [-1,0)$, the set $F_{0,\theta}$ is $C^1$ and satisfies that for any $p\in C_\theta$, $p$ is linear point if and only if the outward normal vector of $C_\theta$ at $p$ belongs to $\Rset \Zset^2$. 

  To prove  $F_{0,\theta}=\partial F_0$, it is enough  to prove that for $\theta_1>\theta_2\in [-1,0]$
$$
\overline H_{\theta_1}(p)=0\quad \Rightarrow \quad \overline H_{\theta_2}(p)=0.
$$
In fact,  fix $p_0\in F_{0,\theta}$. Then $\overline H_{\theta_2}(p_0)\leq 0$. Let  $\mu$ a Mather measure associated with $p_0$ and the Hamiltonian ${1\over 2}|p|^2+V_{\theta_1}$.  Due to the above discussion, orbits on the support of $\mu$ can not enter $B_{1\over 4}(0)+\Zset^2$. Hence, by (\ref{eq:partoverlap}), $\mu$ is also flow invariant with respect the Hamiltonian ${1\over 2}|p|^2+V_{\theta_2}$ and 
$$
\int_{\mu}{1\over 2}|q|^2-V_{\theta_2}-p_0\cdot q\,d\mu=\int_{\mu}{1\over 2}|q|^2-V_{\theta_1}-p_0\cdot q\,d\mu=0.
$$
Therefore, thanks to (\ref{negative-hbar}), 
$$
-\overline H_{\theta_2}(p_0)\leq 0. 
$$
Accordingly,  $\overline H_{\theta_2}(p_0)=0$, which implies that $\mu$ is also a Mather measure associated  with $p_0$ and the Hamiltonian ${1\over 2}|p|^2+V_{\theta_2}$. Moreover, $\partial \overline {H}_\theta(p)\subset \partial \overline {H}_0(p)$ follows immediately from $\overline {H}_\theta\leq \overline {H}_0$ and $\overline {H}_\theta(p_0)=\overline {H}_0(p_0)$.\qed
\begin{center}
\includegraphics[scale=0.6]{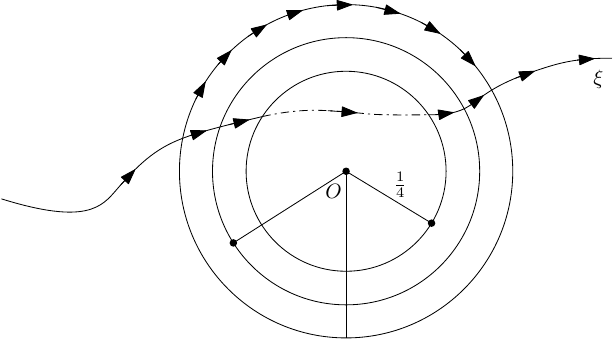}
\captionof{figure}{}
 \label{f13}
\end{center}

\begin{example}\label{example:separable} Finally, we present an example which does not have nonlinear points, which is probably the only computable example. Let 
$$
V(x)=h_1(x_1)+h_2(x_2).
$$
Here $h_1$ and $h_2$ smooth periodic function on $\Rset$ satisfying that 
$$
\max_{\Rset}h_1=\max_{\Rset}h_2=0.
$$
Then for $p=(p_1,p_2)$, 
$$
\overline H(p)=\overline H_1(p_1)+\overline H_2(p_2)
$$
For $i=1,2$, $\overline H_i\geq 0$ is the 1d effective Hamiltonian associated with ${1\over 2}|p_i|^2+h_i$, which possesses explicit formulas \cite{LPV}
$$
\begin{cases}
\overline H_i(p_i)=0 \quad \text{when $|p_i|\leq \int_{0}^{1}\sqrt{-2h_i(x_i)}\,dx_i$}\\[3mm]
|p_i|=\int_{0}^{1}\sqrt{2(\overline H_i(p_i)-h_i(x_i))}\,dx_i \quad \text{otherwise}.
\end{cases}
$$

Case 1: Neither $h_1$ nor $h_2$ is a constant function.  Then
$$
F_0=\{p\in \Rset^2|\ \overline H(p)=0\}=[-L_1,L_1]\times [-L_2, L_2]
$$
for $L_i= \int_{0}^{1}\sqrt{-2h_i(x_i)}\,dx_i$. In this case, for every $p\in F_0$, 
$$
\mathcal {M}_p=\{0\}.
$$
and for $p=(L_1, L_2)$, $G_p=\{(1,0), (0,1), (1,1)\}$. 

\medskip

Case 2: $h_1\equiv 0$ and $h_2$ is not constantly zero.  Then $\overline H_1(p_1)={1\over 2}|p_1|^2$ and  for $c>0$,
$$
F_c=\{p\in \Rset^2|\ \overline H(p)=c\}
$$
contains exactly two linear segments $\pm {\sqrt 2c}\times [-L_2,L_2]$.\qed

\end{example}

\section{Appendix A: Existence of flat part}

For reader's convenience, we provide a self-contained proof of the existence of flat part of $F_0$ by constructing viscosity subsolutions to the cell problem. See \cite{Con,LPV} for existence results under much more general assumptions. 

\medskip

\begin{theo} Assume that $V:\Rset^n\to \Rset$ is continuous and $\Zset^n$-periodic.  Suppose that $\{x\in \Rset^n|\ V=\max_{\Rset^n}V\}$ is a finite set. Then $0$ is an interior point of $F_0$. 
\end{theo}

\nit Proof: Without loss of generality, we assume that $\max_{\Bbb T^n}V=0$. Since $\overline H(p)\geq 0$, it suffices to show that there exists $r>0$ such that for every unit vector $|p|=1$, 
$$
\overline H(rp)\leq 0.
$$
Let $S=\{x\in \Rset^n|\ V=\max_{\Rset^n}V\}$. Clearly, we could find a   neighborhood $U$ of $S$ such that for every $|p|=1$, there exists a smooth periodic function $w_p$ satisfying that 
$$
Dw_p=-p \quad \text{in $U$ }
$$
and $\max_{\Rset^n}|Dw_p|\leq C$ for a constant $C$ independent of $p$. Accordingly, we may find $r>0$ independent of $p$ such that for all $|p|=1$, 
$$
{1\over 2}|rp+rDw_p|^2+V(x)<0 \quad \text{for $x\notin S$}.
$$
Let $v_{rp}$ be a $\Zset^n$ periodic viscosity solution to the cell problem
$$
{1\over 2}|rp+Dv_{rp}|^2+V=\overline H(rp) \quad \text{in $\Rset^n$}.
$$
Choose $x_p\in \Rset^n$ such that
$$
\max_{\Rset T^n}(v_{rp}-rw_p)=(v_{rp}-rw_p)(x_p).
$$
Owing to the definition of viscosity solution, 
$$
{1\over 2}|rp+rDw_p|^2+V(x_p)\geq \overline H(rp).
$$
This immediately leads to $\overline H(rp)\leq 0$. 
\qed

\section{Appendix B: Technical Lemmas}

Below is a technical lemma  used in the paper. 

\begin{lem}\label{exp-decay-1}  Suppose that $h(t):[0, T]\to (0, \infty)$ is a $C^2$ function satisfying that there exists $c>0$ such that
$$
h''(t)\geq c^2 h(t) \quad \text{all $t\in [0, T]$.}
$$
Then
$$
h(t)\leq h(0){e^{-ct}}+h(T)e^{-c(T-t)};
$$
\end{lem}
Proof: Let 
$$
g(t)=h(0){e^{-ct}}+h(T)e^{-c(T-t)}.
$$
Choose $x_0\in  [0,T]$ such that
$$
h(t_0)-g(t_0)=\max_{t\in [0,T]}(h(t)-g(t)).
$$
If $h(t_0)-g(t_0)>0$, then $t_0\in (0,T)$ and 
$$
h''(t_0)\leq g''(t_0)=c^2g(t_0)<c^2h(t_0).
$$
This is a contradiction. Hence $h(t_0)-g(t_0)\leq 0$. \qed

\begin{cor}\label{exp-decay-2} Suppose that $\xi:[0, \infty)\to \Rset^2$ satisfy that 
$$
\begin{cases}
\ddot \xi=-DV(\xi(t))\\[3mm]
\lim_{t\to +\infty}(\xi(t), \dot \xi(t))=(0,0)
\end{cases}
$$
Then there exists $M, \lambda>0$ such that 
$$
|\xi(t)|+|\xi'(t)|\leq Me^{-\lambda t} \quad \text{for all $t\geq 0$}.
$$
\end{cor}

Proof:  Let 
$$
h(t)=|\xi(t))|^2.
$$
Then when $t$ is large enough, 
$$
h''(t)=2\ddot \xi(t)\cdot \xi(t)+2|\dot \xi(2)|^2\geq -2 DV(\xi(t))\cdot \xi(t)\geq ah(t).
$$
Hence the conclusion follows immediately from Lemma \ref{exp-decay-1} above and 
$$
{1\over 2}|\dot \xi|^2+V(\xi(t))=0  \quad \text{for all $t\geq 0$}.
$$
\qed

\section{Appendix C: Existence and uniqueness of limiting directions}
The following conclusion is very well to experts. For readers' convenience, we present the proof here. 
\begin{lem}\label{lem:vunique} For any orbit $\xi:[0,\infty)\to \Rset^2$ satisfying 
\be\label{eq:orbitcondition}
\ddot \xi=-DW(\xi) \quad \text{and} \quad \lim_{t\to \infty}\xi(t)=(0,0),
\ee
 the limit
$$
\lim_{t\to \infty}{\xi(t)\over |\xi(t)|}
$$
exist (denote by $v_\xi$) and  $v_\xi\in \{v_a,\  -v_a,\  v_b,\  -v_b\}$.  

Also, there exists a unique orbit $\xi_+$ (respectively, $\xi_-$) such that $v_{\xi_+}=v_b$ (respectively, $v_{\xi_-}=-v_b$.
\end{lem}

\begin{proof} After a suitable rotation, we may assume that
$$
W(x_1,x_2)=-{a^2\over 2}x_{1}^{2}-{b^2\over 2}x_{2}^{2}+O(|x|^3) \quad \text{when $x$ is near the origin $O$}.
$$
Let 
$z=(\xi,\dot \xi)^{\top}\in\mathbb{R}^{4}$. 
Then our second–order ODE becomes a first–order system
$\dot z = Az + G(z)$ with
\[
A=
\begin{pmatrix}
0 & 0 & 1 & 0\\
0 & 0 & 0 & 1\\
a^2 & 0 & 0 & 0\\
0 & b^2 & 0 & 0
\end{pmatrix},
\qquad
G(z)=O(|z|^{2}),\quad 0<a<b.
\]
The eigenvalues of $A$ are 
$\pm {a},\, \pm {b}$. By the stable manifold theorem, for $\xi$ near $0$ we have
\be\label{eq:stablemanifold}
\dot \xi=
\begin{pmatrix}
-{a} & 0 \\
0 & -{b}
\end{pmatrix}\xi+F(\xi),
\ee
where $F(\xi)=(F_1,F_2)=O(|\xi|^{2})$. Assume
\[
|F(x)|\leq C|x|^2 \quad \text{for $|x|\leq 1$}.
\]
Fix $\lambda_0$ with ${a}<\lambda_0< \min\{{b},\,2{a}\}$. Choose $r_0\in (0,1)$ such that
\be\label{eq:contractionbound}
|DF|\leq {b-a\over 2(\lambda_0-a)(b-\lambda_0)}  \quad \text{in $B_{r_0}(0)$}.
\ee
Write $\xi(t)=(x_1(t),x_2(t))$.

\medskip
\noindent
\textbf{Claim 1.} If $\sup_{t>0}e^{\lambda _0t}|\xi(t)|=\infty$, then $v_\xi=\pm e_1$.  

\medskip
\emph{Proof of Claim 1.} Without loss of generality suppose
\begin{equation}\label{eq:xicontain}
\xi(\Rset)\subset \overline{B_{r_0}(0)}
\end{equation}
and
\[
|F(\xi(t))|\leq {2a-\lambda_0\over 4}|\xi(t)|.
\]
Set $\omega(t)=\tfrac12|\xi(t)|^2$. Then by (\ref{eq:stablemanifold}), 
\[
\dot \omega(t)\leq -2{a}\,\omega(t)+{2a-\lambda_0\over 2} \omega(t),\qquad t\ge0,
\]
so $\omega(t)\leq Ce^{-\lambda_0 t}$. Hence
\[
\dot \xi=
\begin{pmatrix}
-{a} & 0 \\
0 & -{b}
\end{pmatrix}\xi+R(t),\qquad |R(t)|\le Ce^{-\lambda_0t}.
\]
Thus for any fixed $t_0>0$, $t>t_0$
\be
\begin{array}{ll}
x_1(t)=x_1(t_0)\,e^{-a(t-t_0)}+\int_{t_0}^{t} e^{-a(t-s)}\,R(s)\,ds\\[3mm]
x_2(t)=x_2(t_0)\,e^{-b(t-t_0)}+\int_{t_0}^{t} e^{-b(t-s)}\,R(s)\,ds.
\end{array}
\ee
This implies that
\be\label{eq:abode}
\begin{array}{ll}
|x_1(t)-x_1(t_0)\,e^{-a(t-t_0)}|\leq {C\over \lambda_0-a}(e^{-\lambda_0t_0-a(t-t_0)}-e^{-\lambda_0t})\\[3mm]
|x_2(t)-x_2(t_0)\,e^{-b(t-t_0)}|\leq {C\over b-\lambda_0}(e^{-\lambda_0t}-e^{-\lambda_0t_0-b(t-t_0))})
\end{array}
\ee
Consequently,  $\sup_{t>0}e^{\lambda_0t}|x_2(t)|<\infty$ and by the assumption, $\sup_{t>0}e^{-\lambda_0t}|x_1(t)|=\infty$.  
Analyzing the integral representation of $x_1$, we distinguish two cases:

\medskip

- If $\sup_{t>0}e^{\lambda_0t}x_1(t)=\infty$, then by choosing a suitable large $t_0$ in (\ref{eq:abode}), we have that  $x_1(t)\geq C_0e^{-at}$ for $t\geq t_0$ and a positive constant depending on $t_0$.  Then $v_\xi=e_1$.

\medskip

- If $\sup_{t>0}e^{\lambda_0t}(-x_1(t))=\infty$,   then by choosing a suitable large $t_0$  in (\ref{eq:abode}), we have that  $x_1(t)\leq -C_0e^{-at}$ for $t\geq t_0$ and a positive constant depending on $t_0$.  Then $v_\xi=-e_1$.

\medskip
\noindent
\textbf{Claim 2.} If an orbit $\xi=(x_1(t),x_2(t))$ satisfies (\ref{eq:orbitcondition}) and 
$\sup_{t>0} e^{\lambda_0 t}|\xi(t)|<\infty$, then there exists $t_m\to 0$ such that $x_2(t_m)\neq 0$. 

\medskip

\emph{Proof of Claim 2.} We argue by contradiction. Otherwise, without loss of generality, assume that 
$x_2(t)\equiv 0$ for all $t\geq 0$, $x_1(0)\neq 0$, and 
$|F_1(x_1,0)|\leq \tfrac{\lambda_0-a}{2}|x_1|$. This implies 
$|x_1(t)|\geq |x_1(0)|e^{-\tfrac{\lambda_0+a}{2}t}$. Hence 
$\sup_{t>0} e^{\lambda_0 t}|\xi_\theta(t)|=\infty$, a contradiction. 
\qed

\medskip
\noindent
\textbf{Claim 3.} There exists exactly two orbits, up to translation in time, $\xi_\pm$ satisfying (\ref{eq:orbitcondition}) and
\be\label{eq:expobound}
\sup_{t>0}e^{\lambda_0t}|\xi_\pm(t)|<\infty.
\ee

\emph{Proof of Claim 3.} We first establish the existence.  Define
\[
\Gamma=\left\{\xi\in C^{\infty}([0,\infty),\overline{B_{r_0}(0)}): \ \sup_{t\ge0} e^{\lambda_0t}|\xi(t)|<\infty\right\},
\]
with norm $\|\xi\|_0=\sup_{t\ge0} e^{\lambda_0t}|\xi(t)|$. Then $(\Gamma,\|\cdot\|_0)$ is a Banach space.  Fix $\theta_+\in (0,{r_0\over 4}]$ and 
define the Lyapunov–Perron operator
\[
T_{\theta}(\xi)(t)=\Bigg(
\int_t^{\infty}e^{-{a}(t-s)}F_1(\xi(s))\,ds,\ 
\theta e^{-{b}t}+\int_{0}^{t}e^{-{b}(t-s)}F_2(\xi(s))\,ds
\Bigg).
\]
Thanks to (\ref{eq:contractionbound}),  for $\xi_1,\xi_2\in \Gamma$,
\be\label{eq:T-contraction}
||T_\theta(\xi_1)-T_\theta(\xi_2)||_0\leq {1\over 2}||\xi_1-\xi_2||_0.
\ee
Iterating from $\xi_0(t)=(0,\theta e^{-{b}t})$, one shows inductively that for $\xi_m=T_{\theta}(\xi_{m-1})$ and $m\in \Nset$, 
\[
|\xi_m(t)|\leq 3|\theta| e^{-{b}t},\qquad \text{for all $t\ge0$},
\]
consequently $\xi_m\in \Gamma$ for all $m\geq 1$. So, as $m\to \infty$,  $\xi_m$ will converge to a  fixed point $\xi_{\theta}=(x_1(t), x_2(t))\in \Gamma$ of $T_{\theta}$. Hence  $\xi_{\theta}$ satisfies (\ref{eq:stablemanifold}). Combining with $|F_i|\leq C|x|^2$ for $i=1,2$,
\[ |\xi_{\theta}(t)|\leq 3|\theta| e^{-{b}t}, \quad 
|x_1(t)|\leq C\theta^2 e^{-2{b}t} \quad  \mathrm{and}  \quad 
|x_2(t)-\theta e^{-{b}t}|\leq C\theta^2 e^{-2{b}t}.
\]
Thus 
\be\label{eq:twocurvesdirection}
v_{\xi_{\theta}}=\lim_{t\to \infty}{\xi(t)\over |\xi(t)|}=\operatorname{sign}(\theta)e_2,
\ee
which implies the existence of at least two distinct orbits satisfying (\ref{eq:expobound}).

Finally we verify uniqueness to conclude the proof. Let $\xi_1,\xi_2, \xi_3$  be three orbits satisfying  (\ref{eq:orbitcondition}) and (\ref{eq:expobound}). Without loss of generality, we may assume that $\xi_i(\Rset)\subset B_{r_0}(0)$ for $i=1,2,3$.

Owing to Claim 2, by suitable time translation, we may assume that $0\not=e_2\cdot \xi_1(0)=e_2\cdot \xi_0\in [-{r_0\over 3}, {r_0\over 3}]$. Denote by $\theta=e_2\cdot \xi_1(0)=e_2\cdot \xi_0$. Then  both $\xi_1$ and $\xi_2$ are fixed point of $T_\theta$. Accordingly, $\xi_1=\xi_2$ follows from (\ref{eq:T-contraction}).

Finally, the conclusion of the lemma follows from Claim 1, Claim 3 and (\ref{eq:twocurvesdirection}). 
\end{proof}

Below we will present a simple example to demonstrate how to find $\xi_{\pm}$ in the proof of Claim 3 above. 
\begin{example}\label{eq:uniquenessexample} Fix $\alpha\in \Rset$.  Assume that $\xi=(x_1(t),x_2(t))$ satisfies that 
$$
\begin{cases}
\dot x_1(t)=-x_1(t)+\alpha x_2^2(t)\\
\dot x_2(t)=-2x_2(t)
\end{cases}
$$

Write $x_2(0)=\theta\not=0$. Then $x_1(t)=\left(x_1(0)+{\alpha\theta^2\over 3}\right)e^{-t}-\alpha{\theta^2\over 3}e^{-4t}$ and $x_2(t)=\theta e^{-2t}$. Hence for (\ref{eq:expobound}) to hold for any $\lambda_0\in [1,2]$, we must have $x_1(0)=-{\alpha\theta^2\over 3}$.  Therefore, up to translation in time, there are exactly two eligible orbits
$$
\xi_+(t)=\left({\alpha\over 3}e^{-4t}, e^{-2t}\right) \quad \mathrm{and} \quad \xi_-(t)=\left({\alpha\over 3}e^{-4t}, -e^{-2t}\right)
$$
with $v_{\xi_+}=(0,1)$ and  $v_{\xi_-}=-(0,1)$.
\end{example}

\section*{Acknowledgments}
The authors would like to thank Professor Chongqing Cheng and Professor Wei Cheng for their warm hospitality and many insightful discussions during the authors’ visits to Nanjing University in China over the past several years. The authors also thank Professor Chongchun Zeng for helpful discussions.

\bibliographystyle{plain}

\end{document}